\documentclass{article}
\usepackage[title]{appendix}
\usepackage{fullpage}
\usepackage{amsmath,amssymb,amsthm}
\usepackage{amscd}
\usepackage{bm}
\usepackage{hyperref}
\usepackage{dsfont}
\usepackage{enumerate}
\usepackage{epsfig}
\usepackage{float, graphicx}
\usepackage{latexsym, amsxtra}
\usepackage{mathrsfs}
\usepackage{multicol}
\usepackage[normalem]{ulem}
\usepackage{psfrag}
\usepackage[parfill]{parskip}
\usepackage{stmaryrd}
\usepackage{tikz}
\usepackage[T1]{fontenc}
\usepackage{url}
\usepackage{verbatim}
\usepackage{indentfirst}
\usepackage{tikz-cd}

\usepackage{mathtools}

\makeatletter
\def\thm@space@setup{%
  \thm@preskip=2ex \thm@postskip=2ex
}
\makeatother

\hypersetup{hidelinks}

\newtheorem{thm}{Theorem~}[section] 
\newtheorem{lem}[thm]{Lemma~}

\newtheorem{prop}[thm]{Proposition~}

\newtheorem{de}[thm]{Definition~}
\newtheorem{rmk}[thm]{Remark~}

\newtheorem{cond}[thm]{Condition~}

\newcommand{\CC}{\mathbb{C}}
\newcommand{\ZZ}{\mathbb{Z}}
\newcommand{\RR}{\mathbb{R}}

\newcommand{\PP}{\mathbb{P}}

\newcommand{\XX}{\mathbb{X}}
\newcommand{\QQ}{\mathbb{Q}}

\newcommand{\Prd}{\mathscr{P}}
\newcommand{\F}{\mathcal{F}}
\newcommand{\X}{\mathscr{X}}

\newcommand{\B}{\mathbb{B}}
\newcommand{\Ha}{\mathcal{H}}
\newcommand{\C}{\mathcal{C}}
\newcommand{\D}{\mathbb{D}}
\newcommand{\M}{\mathcal{M}}
\newcommand{\V}{\mathcal{V}}
\newcommand{\cL}{\mathcal{L}}

\newcommand\Aut{\mathrm{Aut}}
\newcommand\ADE{\mathrm{ADE}}
\newcommand\GIT{\mathrm{GIT}}

\newcommand\Sym{\mathrm{Sym}}
\newcommand\SL{\mathrm{SL}}
\newcommand\Res{\mathrm{Res}}

\newcommand\IV{\mathrm{IV}}

\newcommand\PSL{\mathrm{PSL}}
\newcommand\Spec{\mathrm{Spec}}
\newcommand\Pic{\mathrm{Pic}}
\newcommand\PO{\mathrm{PO}}
\newcommand\rRe{\mathrm{Re}}
\newcommand\rIm{\mathrm{Im}}
\newcommand\codim{\mathrm{codim}}
\newcommand\Bl{\mathrm{Bl}}
\newcommand\GL{\mathrm{GL}}
\newcommand\diag{\mathrm{diag}}

\newcommand{\bs}{\backslash}
\newcommand{\dbs}{\bs\!\! \bs}

\title{Moduli Spaces of Symmetric Cubic Fourfolds and Locally Symmetric Varieties}
\vspace{1.2cm}
 \author{Chenglong Yu, Zhiwei Zheng}

 \newcommand{\Addresses}{{
  \bigskip
  \footnotesize

  C.~Yu, \textsc{University of Pennsylvania}\par\nopagebreak
  \textit{E-mail address}: \texttt{yucl18@upenn.edu}

  \medskip

  Z.~Zheng, \textsc{Tsinghua University,
    Beijing, China}\par\nopagebreak
  \textit{E-mail address}: \texttt{zheng-zw14@mails.tsinghua.edu.cn}
}}

\begin{document}
\bibliographystyle{amsalpha}
\setlength{\lineskip}{0.5ex}
\setlength{\parskip}{0.5ex}
\maketitle

\vspace{0.5cm}
\begin{abstract}
In this paper we realize the moduli spaces of cubic fourfolds with specified automorphism groups as arithmetic quotients of complex hyperbolic balls or type $\IV$ symmetric domains, and study their compactifications. Our results mainly depend on the well-known works about moduli space of cubic fourfolds, including the global Torelli theorem proved by Voisin (\cite{voisin1986theoreme}) and the characterization of the image of the period map, proved independently by Looijenga (\cite{looijenga2009period}) and Laza (\cite{laza2009moduli,laza2010moduli}). The key input for our study of compactifications is the functoriality of Looijenga compactifications, which we formulate in the appendix (section \ref{section: appendix}).
\end{abstract}

\setcounter{tocdepth}{1}
\tableofcontents
\vspace{1cm}

\section{Introduction}
\label{section: introduction}
Cubic fourfold is an intensively studied object in algebraic geometry. The remarkable work by Voisin in 1986 (\cite{voisin1986theoreme}) showed the global Torelli theorem for smooth cubic fourfolds. Based on this, Allcock-Carlson-Toledo (\cite{allcock2011moduli}) and Looijenga-Swierstra (\cite{looijenga2007period}) realized the moduli space of smooth cubic threefolds as an arrangement complement in an arithmetic ball quotient. Recently Laza-Pearlstein-Zhang (\cite{laza2017moduli}) realized the moduli space of pairs consisting of a cubic threefold and a hyperplane as an arrangement complement in a type $\IV$ arithmetic quotient. In both cases, the authors studied compactifications of the moduli spaces. In this paper we characterize the moduli spaces of cubic fourfolds with specified automorphism groups, and identify the $\GIT$-compactifications with Looijenga compactifications. This generalizes the two results mentioned above.

Let $\F$ be the normalization of the irreducible subvariety parameterizing smooth cubic fourfolds with specified action by finite group $A$ (see section \ref{section: general setup: hypersurfaces with automorphism} for the setup). Let $n=\dim \F$. Let $X$ be a cubic fourfold in $\F$. Consider the induced action of $A$ on $H^4(X,\CC)$, and let $\zeta$ be the character corresponding to $H^{3,1}(X)$. Denote $H^4(X)_\zeta$ to be the $\zeta$-eigenspace, which admits a natural Hermitian form $h$ induced by the topological intersection pairing on $H^4(X,\ZZ)$ (see section \ref{subsection: T-markings}). Then $h$ has signature $(n^{\prime},2)$ if $\zeta=\overline{\zeta}$; $(n^{\prime},1)$ otherwise (see proposition \ref{proposition: hermitian form}). The first main theorem of the paper is the following:

\begin{thm}[Main Theorem 1]
\label{theorem: main1}
\begin{enumerate}[(i)]
\item We have equality $n^{\prime}=n$.
\item The Hodge structure on $H^4(X)_\zeta$ gives an algebraic isomorphism $\Prd \colon \F\cong \Gamma\bs(\D-\Ha_s)$. Here $\D$ is a complex hyperbolic ball if $h$ has signature $(n,1)$; a type $\IV$ symmetric domain otherwise. The group $\Gamma$ is an arithmetic group acting proper discontinuously on $\D$ and $\Ha_s$ is a $\Gamma$-invariant hyperplane arrangement in $\D$.
\item The period map $\mathscr{P}$ extends naturally to an algebraic isomorphism $\F_1\cong \Gamma\bs(\D-\Ha_*)$, where $\F_1$ is a natural partial completion of $\F$, adding cubic fourfolds with at worst $\ADE$-singularities, and $\Ha_*$ is a $\Gamma$-invariant hyperplane arrangement contained in $\Ha_s$.
\end{enumerate}
\end{thm}

Denote $\overline{\F}$ to be the $\GIT$-compactification of $\F$, see section \ref{subsection: Geometric Invariant Theory for Symmetric Hypersurfaces}. We characterize $\overline{\F}$ via:
\begin{thm}[Main Theorem 2]
\label{theorem: main2}
There is an isomorphism between projective varieties $\overline{\F}\cong \overline{\Gamma\bs \D}^{\Ha_*}$.
\end{thm}
Here $\overline{\Gamma\bs \D}^{\Ha_*}$ is the Looijenga compactification of $\Gamma\bs(\D-\Ha_*)$, see section \ref{subsection: main theorem in appendix} in appendix.

Notice that in \cite{gonzalez2011automorphisms}, smooth cubic fourfolds with prime-order automorphisms are classified and form 13 irreducible subvarieties in the moduli of cubic fourfolds (see section \ref{subsection: classification prime order}). Two of the examples in the list are exactly the cases dealt in \cite{allcock2011moduli}, \cite{looijenga2007period} and \cite{laza2017moduli}.

\vspace{0.5cm}
\noindent\textbf{Structure of the Paper}:
We briefly introduce the main content of each section.

In section \ref{section: general setup: hypersurfaces with automorphism} we introduce the notion of symmetry type, and set up the geometric invariant theory of hypersurfaces with specified symmetry type.

In section \ref{section: review of theory of cubic fourfolds} we review concepts about cubic fourfolds, and introduce the global Torelli theorem which was proved by Voisin (\cite{voisin1986theoreme}).

In section \ref{section: local period map} we define the moduli of T-marked cubic fourfolds, and the local period map for those cubic fourfolds. We show that the local period map is an open embedding and characterize its image. Finally we discuss the global period maps by passing to certain quotients.

In section \ref{section: compactification} we investigate the compactifications of both sides of the period map for symmetric cubic fourfolds, and identify them.

In section \ref{section: examples and conjectures} we give some examples and relate them to the previous works.

In section \ref{section: appendix}, we review Looijenga compactification of an arrangement complement in a complex hyperbolic ball or type $\IV$ domain. We prove functoriality of Looijenga compactifications.

\vspace{3ex}

\noindent\textbf{Convention}: All algebraic varieties are defined over the field of complex numbers. The adjectives \emph{open}, \emph{closed} refer to analytic topology and \emph{Zariski-open}, \emph{Zariski-closed} are used for Zariski topology.
\vspace{3ex}

\vspace{3ex}
\noindent \textbf{\large{Notation}}:
\begin{enumerate}[]
\item $(d,k)$: dimension and degree of a hypersurface
\item $V$: complex vector space of dimension $k+2$
\item $F$: degree $d$ polynomial in $k+2$ variables
\item $X$: degree $d$ $k$-fold; cubic fourfold when $(d,k)=(3,4)$
\item $A$: a finite subgroup of $\SL(V)$, containing $\mu_{k+2}$ the center of $\SL(V)$
\item $\overline{A}$: image of $A$ in $\PSL(V)$
\item $\lambda$: character of $A$ with specified restriction to $\mu_{k+2}$
\item $T$: equivalence class of $(A,\lambda)$, called symmetry type of degree $d$ $k$-fold
\item $\V$: eigenspace of $\Sym^d(V^*)$ corresponding to $(A,\lambda)$
\item $C$: centralizer of $A$ in $\SL(V)$
\item $N$: a reductive group acting on $\V$
\item $\V^{sm}/\V^{ss}$: subspace of smooth/semi-stable points in $\V$
\item $\F$: $\GIT$ quotient of $\V^{sm}$ by $N$
\item $\F^m$: moduli space of cubic fourfolds with $T$-markings
\item $\F_1$: moduli space of cubic fourfolds of type $T$, which admits at worst $\ADE$ singularities
\item $\overline{\F}$: $\GIT$ quotient of $\V^{ss}$ by $N$, which is a compactification of $\F$
\item $\M$: moduli space of smooth cubic fourfolds
\item $\overline{\M}$: $\GIT$ compactification of $\M$
\item $(\Lambda_0)\Lambda$: (primitive) middle cohomology lattice of cubic fourfold
\item $\varphi$: topological intersection pairing
\item $\eta$: square of hyperplane class
\item $\widehat{\D}$: local period domain for cubic fourfolds
\item $\widehat{\Gamma}$: monodromy group of the universal family of smooth cubic fourfolds
\item $\Ha_{\Delta}/\Ha_{\infty}$: $\widehat{\Gamma}$-invariant hyperplane arrangement in $\D$
\item $\zeta$: character of $A$, induced by the action of $A$ on $H^{3,1}(X)$
\item $\Lambda_{\zeta}$: eigenspace of $(\Lambda_0)_{\CC}$ corresponding to the character $\zeta$
\item $\sigma_X/\sigma$: representation of $A$ on $H^4(X,\ZZ)$/$\Lambda$
\item $h_X/h$: Hermitian form on $H^4(X)_{\zeta}/\Lambda_{\zeta}$
\item $\D$: local period domain for cubic fourfolds of symmetry type $T$
\item $\Gamma$: monodromy group for universal family of smooth cubic fourfolds of symmetry type $T$
\item $\XX$: locally Hermitian symmetric variety (used in Appendix)
\item $\Ha_s/\Ha_*$: $\Gamma$-invariant hyperplane arrangements in $\D$
\item $\overline{\Gamma\bs\D}^{\Ha_*}$: Looijenga compactification of $\Gamma\bs(\D-\Ha_*)$
\item $\widetilde{\Prd}$: local period map
\item $\Prd$: global period map
\end{enumerate}

\section{General Setup: Symmetric Hypersurfaces}
\label{section: general setup: hypersurfaces with automorphism}
\subsection{Space of Symmetric Polynomials}
Let $V$ be a complex vector space of dimension $k+2$. Denote $\Sym^d(V^*)$ to be the space of degree $d$ polynomials on $V$. We have the natural action of $\SL(V)$ on $\Sym^d(V^*)$, namely, $g(F)=F\circ g^{-1}$ for $g\in \SL(V)$ and $F\in \Sym^d(V^*)$.

The center of $\SL(V)$ is the group $\mu_{k+2}$ consisting of $(k+2)$-th roots of unity. Let $A$ be a finite subgroup of $\SL(V)$ containing $\mu_{k+2}$ and denote $\overline{A}=A/\mu_{k+2}$ the image of $A$ in $\PSL(V)$. Then $\Sym^d(V^*)$ is a representation of $A$.

Notice that for any $\xi\in \mu_{k+2}$ and $F\in \Sym^d(V^*)$, we have $\xi(F)=\xi^{-d}F$. Let $\lambda\colon A\longrightarrow \CC^{\times}$ be a character of $A$ such that $\lambda |_{\mu_{k+2}}$ sends $\xi\in\mu_{k+2}$ to $\xi^{-d}$. Let $\V_{\lambda}$ be the $\lambda$-eigenspace of $\Sym^d(V^*)$. We write $\V=\V_{\lambda}$ for short. Geometrically, an element in $\V$ determines a degree $d$ hypersurface (not necessarily smooth) in $\PP V$, whose automorphism group contains $\overline{A}$.

Two pairs $(A_1,\lambda_1)$ and $(A_2,\lambda_2)$ are called equivalent if and only if there exists $g\in \SL(V)$ such that $gA_1g^{-1}=A_2$ and $\lambda_1(a_1)=\lambda_2(ga_1g^{-1})$. We call an equivalence class a symmetry type, denoted by $T$. There is a poset structure on the space of symmetry types, namely, $T_2\le T_1$ if $T_1,T_2$ are represented by $(A_1,\lambda_1),(A_2,\lambda_2)$ respectively, such that $A_1\subset A_2$ and $\lambda_1=\lambda_2|_{A_1}$. Notice that the space $\V$ depends on the representative $(A,\lambda)$ of $T$.

For $F\in \V$, we denote $Z(F)$ to be the hypersurface determined by $F$ in $\PP V$. For $X=Z(F)$, we denote $\Aut(X)$ to be the group of elements in $\PSL(V)$ preserving $X$, and $\Aut(F)$ to be the preimage of $\Aut(X)$ in $\SL(V)$. From \cite{matsumura1963automorphisms} (theorem 1 and theorem 2) we have:
\begin{thm}[Matsumura-Monsky]
\label{theorem: finiteness of automorphism groups of smooth}
\begin{enumerate}[(i)]
When $X$ is smooth, $d\ge 3$, $k\ge 2$,
\item  the group $\Aut(X)$ is finite,
\item  if $(d,k)\ne (4,2)$, the group $\Aut(X)$ contains all biregular automorphisms of $X$.
\end{enumerate}
\end{thm}
Apparently, the group $\overline{A}$ embeds into $\Aut(X)$, for any $X=Z(F)$. We propose the following conditions on the symmetry type $T$:

\begin{cond}
\label{condition: smooth}
The linear space $\V$ contains a point $F$ determining smooth hypersurface.
\end{cond}

\begin{cond}
\label{condition: smooth and identification of groups}
The linear space $\V$ contains a point $F$ with the determined hypersurface $X$ smooth and $\overline{A}=\Aut(X)$.
\end{cond}

For $T$ satisfying condition \ref{condition: smooth}, a generic point in $\V$ determines a smooth hypersurface. We have similar result about condition \ref{condition: smooth and identification of groups}:
\begin{prop}
If $T=[(A,\lambda)]$ satisfies condition \ref{condition: smooth and identification of groups}, then a generic element in $\V$ determines a smooth hypersurface $X$ with $\overline{A}=\Aut(X)$.
\end{prop}
\begin{proof}
Suppose $F\in \V$ with $X=Z(F)$ smooth, and $A= \Aut(X)$. Then any small deformation $F_1$ of $F$ in $\V$ determines a smooth hypersurface $Z(F_1)$. By theorem 2.5 in \cite{zheng2017}, when $F_1$ is sufficiently close to $F$, there exists $g\in \PSL(V)$ such that $g\Aut(Z(F_1))g^{-1}\subset \Aut(X)=\overline{A}$. Since $F_1\in \V$, we have $\overline{A}\subset\Aut(Z(F_1))$, hence $\overline{A}= \Aut(Z(F_1))$.
\end{proof}
\subsection{Geometric Invariant Theory for Symmetric Hypersurfaces}
\label{subsection: Geometric Invariant Theory for Symmetric Hypersurfaces}
Now we assume that $d\ge 3$, $k\ge 2$. Given a symmetry type $T=[(A,\lambda)]$ satisfying condition \ref{condition: smooth}, let
\begin{equation*}
C=\{g\in \SL(V)\big{|}gag^{-1}=a,\forall a\in A\}
\end{equation*}
and
\begin{equation*}
N=\{g\in \SL(V)\big{|}gAg^{-1}=A, \lambda(gag^{-1})=\lambda(a),\forall a\in A\}
\end{equation*}
be two reductive subgroups of $\SL(V)$. For reductivity, see \cite{luna1979Chevalley}, lemma 1.1.


\begin{lem}
\label{lemma: action of N and stableness}
There is a natural action of $N$ on $\V$, under which the points in $\V$ defining smooth hypersurfaces are stable.
\end{lem}
\begin{proof}
For any $g\in N$ and $F\in \V$, we need to show $g(F)\in \V$. For any $a\in A$, we have:
\begin{align*}
a(g(F))=g(g^{-1}ag(F))=g(\lambda(g^{-1}ag)F)=g\lambda(a)F=\lambda(a)g(F),
\end{align*}
which implies $g(F)\in \V$ by definition of $\V$. Therefore, there is a natural action of $N$ on $\V$.

Now take $F\in \V$ with $X=Z(F)$ smooth. Then $\Aut(X)$ is finite by theorem \ref{theorem: finiteness of automorphism groups of smooth}. Since the stabilizer group of $F$ under action of $N$ is a subgroup of $\Aut(F)$, hence also finite. Moreover, $NF$ is closed in $\SL(V)F$, and the latter is closed in $\Sym^d(V^*)$ since $F$ is smooth. Thus $NF$ is closed in $\Sym^d(V^*)$, hence also closed in $\V$. We conclude that $F$ is stable under the action of $N$.
\end{proof}

Denote $\V^{sm}=\{F\in \V\big{|}Z(F)$ smooth$\}$, $\V^{ss}$ the set of semi-stable elements in $\V$ under the action of $N$, and $\PP\V^{sm}$, $\PP\V^{ss}$ their projectivizations. By lemma \ref{lemma: action of N and stableness}, we can take $\F=N\dbs \PP \V^{sm}$ to be the $\GIT$ quotient, with compactification $\overline{\F}=N\dbs\PP\V^{ss}$. Different representatives of the symmetry type induce canonically isomorphic $\GIT$-quotients. Define $\M=\SL(V)\dbs \PP\Sym^d(V^*)^{sm}$ to be the moduli space of smooth degree $d$ hypersurfaces in $\PP(V)$, with compactification $\overline{\M}=\SL(V)\dbs \PP\Sym^d(V^*)^{ss}$. We have the following proposition:
\begin{prop}
\label{prop: git normalization}
There is a natural morphism $j\colon \overline{\F}\longrightarrow \overline{\M}$ sending $[F]\in \F$ to $[F]\in \M$ for any $F\in \V^{sm}$. This morphism is finite. When $T$ satisfies condition \ref{condition: smooth and identification of groups}, the morphism $j$ is a normalization of its image.
\end{prop}
\begin{proof}
Here we use a projective version of the main theorem in \cite{luna1975Adh}. See the argument of proposition $8$ in \cite{ressayre2010geometric}.  Since $A$ is a finite group, there exists certain symmetric power $\Sym^l (\V)$ on which the $A$-action is trivial. Consider the $\SL(V)$-action on the coordinate ring $\bigoplus_{m}\Sym^{lm} (\Sym^d(V^*)^*)$ of $(\PP(\Sym^d(V^*)), \mathcal{O}(l))$. Notice that $N$ is of finite index in the normalizer of $A$ in $\SL(V)$. By the main theorem in \cite{luna1975Adh}, we have a finite morphism
\begin{equation*}
\widetilde{j}\colon\Spec((\bigoplus_{m}\Sym^{lm}(\V^*))^N)\longrightarrow\Spec((\bigoplus_{m}\Sym^{lm}(\Sym^d(V^*)^*))^{\SL(V)})
\end{equation*}
sending semi-stable points to semi-stable points, and preserving the cone structures. Thus $\widetilde{j}$ does not contract any line, so descends to a finite morphism $j\colon \overline{\F}\longrightarrow \overline{\M}$. The morphism $j$ sends $[F]\in \F$ to $[F]\in \M$ for any $F\in \V^{sm}$.

We claim that when $T$ satisfies condition \ref{condition: smooth and identification of groups}, the morphism $j$ is generically injective. Take generically $F_1,F_2\in \V$ and assume $[F_1]=[F_2]$ in $\M$. Then there exists $g\in \SL(V)$ with $g(F_1)=F_2$. By the calculation
\begin{align}
\label{equation: g-1hg}
g^{-1}ag(F_1)=g^{-1}a(F_2)=g^{-1}\lambda(a)F_2=\lambda(a)F_1
\end{align}
we have that $g^{-1}ag\in \SL(V)$ is an automorphism of $Z(F_1)$. By the genericity of $F_1$, we have $A\cong \Aut(F_1)$, which implies that $g^{-1}ag\in A$. Then by equation \eqref{equation: g-1hg} and $F_1\in \V$, we have $\lambda(g^{-1}ag)=\lambda(a)$. This implies that $g\in N$, hence $[F_1]=[F_2]$ in $\F$. Thus $j$ is generically injective.

Moreover, since $\overline{\F}$ is normal and projective, $j$ is a normalization of its image.
\end{proof}

Let $T=[(A,\lambda)]$ be a symmetry type satisfying condition \ref{condition: smooth}. Consider the automorphism groups $\Aut(F)$ for all $F\in \V_T^{sm}$. There exists $F^{\prime}\in \V_T^{sm}$ such that $\#\Aut(F)$ is minimal. Fix this polynomial $F^{\prime}$, and denote $A^{\prime}=\Aut(F^{\prime})$. We have a symmetry type $T^{\prime}=[(A^{\prime},\lambda^{\prime})]$, where $a(F^{\prime})=\lambda^{\prime}(a)F$ for all $a\in A^{\prime}$. We have $T\ge T^{\prime}$, and $T^{\prime}$ satisfies condition \ref{condition: smooth and identification of groups}. For $T^{\prime}$, there are the corresponding $N^{\prime}, \V^{\prime}$ and $\overline{\F^{\prime}}$. We have the following proposition:
\begin{prop}
\label{proposition: nonsaturated and saturated}
There exists a natural finite morphism $\overline{\F}\longrightarrow\overline{\F^{\prime}}$.
\end{prop}
\begin{proof}
By proposition \ref{prop: git normalization}, we have two finite morphisms $j\colon \overline{\F}\longrightarrow \overline{\M}$ and $j^{\prime}\colon \overline{\F^{\prime}}\longrightarrow \overline{\M}$, and the latter one is a normalization of its image. We show that $j$ and $j^{\prime}$ have the same image. Firstly, we have that $j^{\prime}(\overline{\F^{\prime}})\subset j(\overline{\F})$ since $\V^{\prime}\subset\V$. By lemma 2.5 in \cite{zheng2017}, when $F^{\prime\prime}\in \V$ is sufficiently close to $F^{\prime}$, there exists $g\in \SL(V)$, such that $g\Aut(F^{\prime\prime})g^{-1}\subset\Aut(F^{\prime})=A^{\prime}$. By minimality of $\# A^{\prime}$, we have $g\Aut(F^{\prime\prime})g^{-1}= A^{\prime}$. This implies that $\Aut(g(F^{\prime\prime}))=A^{\prime}$, hence $g(F^{\prime\prime})\in \V^{\prime}$. We then have that $\dim (j(\overline{\F}))\le \dim (j^{\prime}(\overline{\F^{\prime}})$. By irreducibilities of the two images, they are the same.

By universal property of normalization, the morphism $j$ factors through $j^{\prime}$. Therefore, we have naturally a finite morphism $\overline{\F}\longrightarrow \overline{\F^{\prime}}$.
\end{proof}

\begin{rmk}
The fiber of the finite morphism $\overline{\F}\longrightarrow \overline{\F^{\prime}}$ over $[F^{\prime}]$ is bijective to the orbit of $(A,\lambda)$ in the set of subdatas of $(A^{\prime},\lambda^{\prime})$ under the action of $N^{\prime}$.
\end{rmk}
\subsection{Universal Deformation}
\label{universal defo}
We fix a type $T=[(A,\lambda)]$ satisfying condition \ref{condition: smooth}, and assume $d\ge 3$ and $k\ge 2$. Next we use Luna's \'etale slice theorem to describe the local structure of $\F$, and construct the universal family of smooth degree $d$ $k$-folds of type $T$. We follow the argument in \cite{zheng2017} (section 2). For Luna's \'etale slice theorem and its proof, one can refer to \cite{luna1973slices} or \cite{popov1994invariant}.

Denote $G$ to be the centralizer of $\overline{A}$ in $\PSL(V)$, which acts on the affine variety $\PP\V^{sm}$. For any $x\in \PP\V^{sm}$, we denote $Gx$ to be the orbit of $x$ and $G_x$ to be the stabilizer of $x$. By lemma \ref{lemma: action of N and stableness}, $Gx$ is closed in the affine variety $\PP\V^{sm}$ and $G_x$ is finite. By Luna's \'etale slice theorem, there exists a smooth, locally closed, $G_x$-invariant subvariety $S$ containing $x$, such that:
\begin{enumerate}[(i)]
\item The image of $\kappa\colon G\times^{G_x}S\longrightarrow \PP\V^{sm}$, denoted by $U$, is Zariski-open and $G$-invariant,
\item The morphism $\kappa\colon G\times^{G_x}S\longrightarrow U$ is \'etale,
\item The morphism $G\dbs\kappa\colon G_x\dbs S\longrightarrow G\dbs U$ is \'etale,
\item The above two morphisms induce an isomorphism
\begin{equation}
\label{equation: luna isomorphism}
G\times^{G_x}S\cong U\underset{G\dbs U}{\times}G_x\dbs S.
\end{equation}
\end{enumerate}

We can shrink $S$ in the analytic category such that:
\begin{enumerate}[(i)]
\item[(v)] $S$ is $G_x$-invariant, contractible and contains $x$, with $U=\kappa(G\times^{G_x}S)$ a $G$-invariant open subset of $\PP\V^{sm}$,
\item[(vi)] the morphism between analytic spaces: $G_x\dbs S\longrightarrow G\dbs U$ is an isomorphism.
\end{enumerate}

From \eqref{equation: luna isomorphism}, we have an isomorphism between analytic spaces:
\begin{equation*}
G\times^{G_x}S\cong U,
\end{equation*}
by which we have a principal $G_x$-bundle $G\times S\longrightarrow U$. In particular, $G\times S$ covers $U$.
\begin{de}
For any symmetry type $T$, we define a category $\C_{d,k}^T$ as follows. The objects are families of degree $d$ $k$-folds of type $T$ with a specified central fiber, and the morphisms are holomorphic maps between families, sending central fiber to central fiber and compatible with the action of $\overline{A}$.
\end{de}
\begin{prop}
\label{prop: universal deformation}
The family $\X_S$ of degree $d$ $k$-folds of type $T$ over $S$ has the following universal property. For any subfamily $\X_{S^{\prime}}\longrightarrow S^{\prime}\subset U$ of degree $d$ $k$-folds of type $T$ containing a central fiber $X^{\prime}$ with isomorphism $f\colon X^{\prime}\cong X$ commuting with $\overline{A}$, we have a unique morphism in the category $\C_{d,k}^T$:
\begin{equation*}
\begin{CD}
\X_{S^{\prime}} @>\widetilde{f}>> \X_S\\
@VVV               @VVV\\
S^{\prime}               @>>> S
\end{CD}
\end{equation*}
such that the restriction of $\widetilde{f}$ to $X^{\prime}$ is $f$.

Moreover, for any two fibers $X_1, X_2$ of $\X_S$ with an isomorphism $g\colon X_1\longrightarrow X_2$ commuting with $\overline{A}$, we can extend $g$ uniquely to a morphism $\widetilde{g}\colon \X_S\longrightarrow \X_S$ in $\C_{d,k}^T$
\end{prop}

\begin{proof}[Proof of proposition \ref{prop: universal deformation}]
The base $S^{\prime}$ lies in $U$ and is covered by $G\times S$. Thus we have a unique lifting $S^{\prime}\hookrightarrow G\times S$, sending $x^{\prime}$ to $(f^{-1}, x)$. In other words, we have uniquely $\widetilde{f}\colon \X_{S^{\prime}}\longrightarrow \X_S$, which restricts to $f$ on $X^{\prime}$.

Now suppose $X_1$, $X_2$ are two fibers of $\X_S$ with isomorphism $g\colon X_1\cong X_2$. Denote $x_1$, $x_2$ the corresponding base points in $S$. Then $(g,x_1),(id,x_2)\in G\times S$ have the same image in $U$. Since $G\times S\longrightarrow U$ is a principal $G_x$-bundle, the two pairs $(g,x_1)$ and $(id,x_2)$ are $G_x$-equivalent, hence $g\in G_x$. The corollary follows.
\end{proof}

We have the following lemma, which will be used in the proof of proposition \ref{proposition: properties of local period map}. Since it holds for general  degree $d$ $k$-folds, we state and prove it here.
\begin{lem}
\label{lemma: glue automorphisms}
Let
\begin{equation*}
\begin{tikzcd}
\X \arrow[hook]{r}\arrow{d} &S\times\PP V\arrow{dl} \\
S
\end{tikzcd}
\end{equation*}
be a family of smooth degree $d$ $k$-folds, with the base $S$ contractible. Suppose there is a group $\widetilde{A}$, such that for all $s\in S$, the fiber $\X_s$ admits a biregular action of $\widetilde{A}$, with induced actions on $H^n(\X_s,\ZZ)$ compatible with respect to the local trivialization. Then there is an action of $\widetilde{A}$ on the whole family $\X\longrightarrow S$ inducing on each fiber the existed action.
\end{lem}

We need another lemma from \cite{javanpeykar2017complete} (proposition 2.12) and \cite{matsumura1963automorphisms}:
\begin{lem}
\label{lemma: faithful of action on cohomology}
For $d\ge 3$, $k\ge 2$, and a smooth degree $d$ $k$-fold $X$, the induced action of $\Aut(X)$ on $H^k(X,\ZZ)$ is faithful.
\end{lem}

\begin{proof}[Proof of lemma \ref{lemma: glue automorphisms}]
Take any $s\in S$. By theorem 2.5 in \cite{zheng2017}, there is a universal hypersurface family $\X^{\prime}$ of $\X_s$, such that any isomorphism between two fibers (may coincide) of $\X^{\prime}$ comes from an automorphism of the central fiber $\X_s$. There exists an open neighbourhood $U$ of $s$ in $S$, with a unique morphism $\X|_U\longrightarrow \X^{\prime}$. Then for any $s^{\prime}\in U$, the action of $\widetilde{A}$ on $\X_{s^{\prime}}$ is induced by a subgroup $\widetilde{A}^{\prime}$ of $\Aut(\X_s)$. By lemma \ref{lemma: faithful of action on cohomology}, and compatibility of induced action of $\widetilde{A}$ on $\X_s$ and $\X_{s^{\prime}}$, we have that $\widetilde{A}=\widetilde{A}^{\prime}$ as subgroups of $\Aut(\X_s)$. Therefore, the actions of $\widetilde{A}$ on fibers of $\X\longrightarrow S$ glue to an action of $\widetilde{A}$ on the whole family.
\end{proof}

\section{Review: Period Map for Smooth Cubic Fourfolds}
\label{section: review of theory of cubic fourfolds}
In this section we recall some fundamental facts on period map for cubic fourfolds, the main references are \cite{voisin1986theoreme}, \cite{hassett}, \cite{looijenga2009period}, \cite{laza2009moduli,laza2010moduli}.

Take $(d,k)=(3,4)$. Then we have $\M$ the moduli of smooth cubic fourfolds, as a Zariski-open subset of its $\GIT$ compactification $\overline{\M}$. Let $X$ be a smooth cubic fourfold. We denote $\varphi_X$ to be the intersection pairing on $H^4(X, \ZZ)$. Then $(H^4(X,\ZZ),\varphi_X)$ is an odd unimodular lattice of signature $(21,2)$. Denote $\eta_X$ to be square of the hyperplane class of $X$. Then $H^4_0(X, \ZZ)=\eta_X^{\perp}$ is an even sublattice of discriminant $3$. Now we define $(\Lambda, \Lambda_0, \eta)$ to be an abstract data isomorphic to $(H^4(X, \ZZ), H^4_0(X, \ZZ), \eta_X)$, this does not depend on the choice of the cubic fourfold $X$.

\begin{de}
A marking of the cubic fourfold $X$ is an isomorphism $\Phi\colon H^4(X, \ZZ)\cong\Lambda$ sending $\eta_X$ to $\eta$.
\end{de}

Two marked cubic fourfolds $(X_1, \Phi_1)$ and $(X_2, \Phi_2)$ are called equivalent if there exists a linear isomorphism $g\colon X_1\longrightarrow X_2$ such that $\Phi_1=g^*\Phi_2$. Let $\M^m$ be the set of equivalence classes of marked cubic fourfolds. From \cite{zheng2017}, section 3, we have:

\begin{prop}
The set $\M^m$ is a complex manifold in a natural way.
\end{prop}

Next we define the period domain and period map for cubic fourfolds. Let
\begin{equation*}
\widetilde{\D}\coloneqq\PP\{x\in (\Lambda_0)_{\CC}\big{|}\varphi(x,x)=0, \varphi(x, \overline{x})<0\}.
\end{equation*}
this is an analytically open subset of a quadric hypersurface in $\PP(\Lambda_0)_{\CC}$, and has two connected components. We have naturally a holomorphic map
\begin{equation*}
\widetilde{\Prd}\colon \M^m\longrightarrow \widetilde{\D}
\end{equation*}
sending $(X, \Phi)\in \M^m$ to $\Phi(H^{3,1}(X))$. It is called the local period map for cubic fourfolds.

Let $\widehat{\D}$ be one connected component of $\widetilde{\D}$ and $\widehat{\Gamma}$ the index $2$ subgroup of $\Aut(\Lambda,\varphi, \eta)$ which respects the component $\widehat{\D}$. Then $\widehat{\Gamma}$ is an arithmetic group acting on $\widehat{\D}$ and $\widetilde{\Prd}$ descends to
\begin{equation*}
\Prd\colon \M\longrightarrow \widehat{\Gamma}\bs\widehat{\D},
\end{equation*}
which is called the (global) period map for cubic fourfolds.

\begin{rmk}
The subgroup $\widehat{\Gamma}$ consists of elements in $\Gamma$ with spinor norm 1. Since there exist vectors in $\Lambda_0$ with self intersection $-2$, the group $\widehat{\Gamma}$ is of index $2$ in $\Aut(\Lambda, \varphi,\eta)$.
\end{rmk}

The global Torelli theorem is originally proved by Voisin (\cite{voisin1986theoreme}), with an erratum (\cite{voisin2008erratum}) based on some work by Laza (\cite{laza2009moduli}):
\begin{thm}[Voisin]
\label{theorem: global Torelli for cubic fourfolds}
The period map $\Prd$ is an open embedding.
\end{thm}
\begin{rmk}
In fact, the period map $\Prd$ is algebraic, see discussion in \cite{hassett2000special} (proposition 2.2.3).
\end{rmk}
We give a lemma which will be constantly used. See \cite{zheng2017} (theorem 1.1).
\begin{lem}
\label{lemma: isomotphism between automorphism groups}
Take $X$ a smooth cubic fourfold, then $\Aut(X)\cong \Aut(H^4(X,\ZZ),\varphi_X, \eta_X, H^{3,1}(X))$.
\end{lem}

We have a refined version of theorem \ref{theorem: global Torelli for cubic fourfolds}:
\begin{prop}[Voisin, Hassett, Looijenga, Laza]
\label{prop: strong global torelli for cubic fourfolds}
The local period map $\widetilde{\Prd}$ is an open embedding, with image being the complement of a hyperplane arrangement invariant under the action of $\Aut(\Lambda, \eta)$ on $\widetilde{D}$.
\end{prop}

\begin{proof}
Combining theorem \ref{theorem: global Torelli for cubic fourfolds} and lemma \ref{lemma: isomotphism between automorphism groups} we have injectivity. The characterization of the image of $\widetilde{\Prd}$ is due to Looijenga (\cite{looijenga2009period}) and Laza (\cite{laza2010moduli} (theorem 1.1), more precise version will be discussed in proposition \ref{proposition: image of local period map in general}.
\end{proof}

\section{Period Maps for Symmetric Cubic Fourfolds}
\label{section: local period map}
\subsection{Local Period Map for Symmetric Cubic Fourfolds}
\label{subsection: T-markings}
In this section we are going to discuss the local and global period maps for symmetric cubic fourfolds. Let $(d,k)=(3,4)$, and fix a symmetry type $T=[(A,\lambda)]$ satisfying condition \ref{condition: smooth}. We first introduce the local period domains with action of arithmetic groups. Let $X=Z(F)$ for a generic point $F\in \V$. Recall that the action of $A$ on $X$ induces an action of $A$ on $H^{3,1}(X)$. This action is a character $\zeta\colon A\longrightarrow\CC^{\times}$ with trivial restriction on $\mu_{k+2}$. We denote
\begin{equation*}
H^4(X)_{\zeta}=\{x\in H^4(X)\big{|}ax=\zeta(a)x, \forall a\in A\}.
\end{equation*}
Define a Hermitian form $h\colon H^4(X)_{\zeta}\times H^4(X)_{\zeta}\longrightarrow \CC$ by $h(x,y)=\varphi(x,\overline{y})$. Denote $\sigma_X$ to be the action of $A$ on $H^4(X,\ZZ)$. Let $\sigma$ be a action of $A$ on $\Lambda$, making $(\Lambda, \eta, \sigma)$ isomorphic to $(H^4(X, \ZZ), \eta_X, \sigma_X)$. Denote $\Lambda_{\zeta}\subset\Lambda_0\otimes\CC$ to be the $\zeta$-eigenspace of the action of $A$ on $(\Lambda_0)_{\CC}$.

\begin{prop}
\label{proposition: hermitian form}
The Hermitian form $h$ has signature $(n^{\prime},2)$ if $\zeta=\overline{\zeta}$ (this is also equivalent to $\zeta(A)\subset\mu_2$); it has signature $(n^{\prime},1)$ otherwise. Here $n^{\prime}$ is a non-negative integer independent of the choice of $X$.
\end{prop}

\begin{proof}
Notice that the lattice $H^4(X,\ZZ)$ has signature $(21,2)$, with negative part $H^{3,1}(X)\oplus H^{1,3}(X)$. If $\zeta(A)$ is not contained in $\mu_2$, we have $\zeta\ne\overline{\zeta}$. Since $H^{1,3}$ lies in $\overline{\zeta}$-eigenspace, the signature of $h$ is $(n^{\prime},1)$.

For the case $\zeta(A)\subset \mu_2$, both $H^{3,1}(X)$ and $H^{1,3}(X)$ are contained in $H_{\zeta}$, hence $h$ has signature $(n^{\prime}, 2)$.
\end{proof}

An isomorphisms $\Phi\colon (H^4(X, \ZZ), \eta_X, \sigma_X)\cong(\Lambda, \eta, \sigma)$ is called a T-marking of $X$. We consider pairs consisting of a smooth cubic fourfold and its T-marking. Two such pairs $(X_1, \Phi_1)$ and $(X_2, \Phi_2)$ are equivalent if there exists $g\in G$ such that $\Phi_1=g^*\Phi_2$. Let $\F^m$ be the set of equivalence classes of such pairs, we have:

\begin{prop}
\label{proposition: moduli of A-marked cubic fourfolds}
The set $\F^m$ is naturally a complex manifold.
\end{prop}

\begin{proof}
First we describe the local charts on $\F^m$. Take a point $(X,\Phi)\in \F^m$, and take a universal deformation $\X_S\longrightarrow S$ of $X$ as in proposition \ref{prop: universal deformation}. Since $S$ is contractible, the local system $R^4\pi_*(\ZZ)$ is trivializable over $S$ and the T-marking $\Phi$ of $X$ naturally extends to T-marking of every fiber of $\X_S\longrightarrow S$. Thus we have a map
\begin{equation*}
\alpha\colon S\longrightarrow \F^m.
\end{equation*}
We first show that $\alpha$ is injective. Suppose $X_1$, $X_2$ are two fibers of $\X_S$, with $\Phi_1, \Phi_2$ the induced T-markings by $\Phi$ respectively, such that $(X_1,\Phi_1)$ and $(X_2,\Phi_2)$ represent the same point in $\F^m$. Then there exists $g\colon X_1\cong X_2$ with $\Phi_2=\Phi_1\circ g^*$. By proposition \ref{prop: universal deformation} we have $g\in G_x$ and $\Phi=\Phi\circ g^*$, hence $g^*=id$. By lemma \ref{lemma: isomotphism between automorphism groups} we have $g=id$. Thus $\alpha$ is injective.

By definition, $\F^m$ is covered by countably many such $\alpha(S)$, and they form a basis of a topology. To show $\F^m$ is a complex manifold, we need to prove that the topology is Hausdorff. Suppose not, then we have two non-separated points $(X,\Phi), (X^{\prime},\Phi^{\prime})\in \F^m$. Then $X$ and $X^{\prime}$ are isomorphic (because $\F$ is separated). Without loss of generality, we just assume $X^{\prime}=X$. Take $\X_S\longrightarrow S$ the universal family as in proposition \ref{prop: universal deformation}, and
\begin{equation*}
\alpha,\alpha^{\prime}\colon S\longrightarrow \F^m
\end{equation*}
induced by $\Phi$ and $\Phi^{\prime}$. Now since $(X,\Phi)$ and $(X^{\prime}, \Phi^{\prime})$ are non-separated, we have $\alpha(S)\cap \alpha^{\prime}(S)\ne \varnothing$. Thus there exists $x_1\in S$ with corresponding cubic fourfold $X_1$, such that the two pairs $(X_1,\Phi)$ and $(X_1,\Phi^{\prime})$ represent the same point in $\F^m$. Then there is an automorphism $g$ of $X_1$, such that $\Phi^{\prime}=\Phi\circ g^*$. Proposition \ref{prop: universal deformation} implies that $g$ is also an automorphism of $X$ and satisfies the above relation. Thus $(X,\Phi)=(X,\Phi^{\prime})$ in $\F^m$, contradiction. We showed the Hausdorff property, hence conclude that $\F^m$ is naturally a complex manifold.
\end{proof}

\begin{rmk}
Proposition \ref{proposition: moduli of A-marked cubic fourfolds} can be generalized to degree $d$ $k$-folds ($d\ge 3$, $k\ge2$) with specified automorphism group. The argument is the same.
\end{rmk}

When $h$ has signature $(n^{\prime},1)$, we define $\D_T=\PP\{x\in \Lambda_{\zeta}\big{|}\varphi(x,\overline{x})<0\}$, which is a hyperbolic complex ball of dimension $n^{\prime}$; when $h$ has signature $(n^{\prime},2)$, define $\D_T$ to be a component of $\PP\{x\in (\Lambda_0)_{\zeta}\big{|}\varphi(x, x)=0, \varphi(x,\overline{x})<0\}$, which is a type $\IV$ symmetric domain of dimension $n^{\prime}$.

We define local period map for symmetric cubic fourfolds of type $T$ as the map from $\F^m$ to $\D_T\sqcup \overline{\D_T}$, sending $(X,\Phi)$ to $\Phi(H^{3,1}(X))$, still denoted by $\widetilde{\Prd}$. We make the choice of $\D_T$ such that $\widetilde{\Prd}$ has nonempty image in $\D_T$. Write $\D=\D_T$ if there is no confusion.
\subsection{Properties of Local Period Maps for Symmetric Cubic Fourfolds}
We need to review basic works by Laza (\cite{laza2009moduli,laza2010moduli}). In \cite{laza2009moduli} Laza classified stable and semistable cubic fourfolds. One of the main theorems is:
\begin{thm}[\cite{laza2009moduli}]
\label{theorem: laza ADE stable}
A cubic fourfold with at worst $\ADE$-singularities is stable.
\end{thm}

Laza proved that the period map $\Prd\colon \M\longrightarrow \widehat{\Gamma}\bs\widehat{\D}$ extends to the moduli space $\M_1$ of cubic fourfolds with at worst $\ADE$ singularities, and characterized its image. The results are gathered in the following theorem:
\begin{thm}[\cite{laza2010moduli}]
\label{theorem: laza extension}
The period map $\Prd\colon \M\longrightarrow \widehat{\Gamma}\bs \widehat{\D}$ has image $\widehat{\Gamma}\bs (\widehat{\D}-\Ha_{\infty}-\Ha_{\Delta})$, and extends holomorphically to
\begin{equation*}
\Prd\colon \M_1\longrightarrow \widehat{\Gamma}\bs\widehat{\D}
\end{equation*}
with image $\widehat{\Gamma}\bs(\widehat{\D}-\Ha_{\infty})$.
Here $\Ha_{\infty}, \Ha_{\Delta}$ are two $\widehat{\Gamma}$-invariant hyperplane arrangements in $\widehat{\D}$, with the quotients $\widehat{\Gamma}\bs\Ha_{\infty}$ and $\widehat{\Gamma}\bs\Ha_{\Delta}$ irreducible.
\end{thm}

\begin{rmk}
This characterization of the image $\Prd(\M)$ was conjectured by Hassett in \cite{hassett2000special}. Hassett defined the special cubic fourfolds, some of which correspond to polarized $K3$ surfaces. The hyperplane arrangements $\Ha_{\Delta}$ and $\Ha_{\infty}$ are two particular ones, parameterizing nodal cubic fourfolds and secent lines of determinantal cubic fourfold, and corresponding to $K3$ surfaces of degree 6 and 2 respectively. See \cite{hassett2000special}, section 4.2 and 4.4.
\end{rmk}

We have also the following marked-version of theorem \ref{theorem: laza extension}:
\begin{prop}
\label{proposition: image of local period map in general}
The local period map $\widetilde{\Prd}\colon \M^m\longrightarrow \widetilde{\D}$ has image $\widetilde{\D}-\Ha_{\infty}-\Ha_{\Delta}-\overline{\Ha_{\infty}}-\overline{\Ha_{\Delta}}$.
\end{prop}
\begin{proof}
By theorem \ref{theorem: laza extension}, the image of $\widetilde{\Prd}$ lies in $\widetilde{\D}-\Ha_{\infty}-\Ha_{\Delta}-\overline{\Ha_{\infty}}-\overline{\Ha_{\Delta}}$. Take any point $x$ in $\widetilde{\D}-\Ha_{\infty}-\Ha_{\Delta}-\overline{\Ha_{\infty}}-\overline{\Ha_{\Delta}}$. By theorem \ref{theorem: laza extension} the point $[x]\in \widehat{\Gamma}\bs(\widehat{\D}-\Ha_{\infty}-\Ha_{\Delta})$ lies in the image of $\Prd\colon\M\longrightarrow{\widehat{\Gamma}}\bs\widehat{\D}$. Thus the orbit $\Aut(\Lambda,\eta) x$ intersects with $\widetilde{\Prd}(\M^m)$. Notice that the set $\widetilde{\Prd}(\M^m)$ is $\Aut(\Lambda,\eta)$-invariant, hence contains the orbit $\Aut(\Lambda,\eta) x$. We showed the surjectivity.
\end{proof}

For a specified type $T$, we have a natural embedding $\D\sqcup\overline{\D}\hookrightarrow \widetilde{\D}$. Denote $\Ha_s=\D\cap (\Ha_{\Delta}\cup\Ha_{\infty}\cup\overline{\Ha_{\Delta}}\cup\overline{\Ha_{\infty}})$ and $\Ha_{*}=\D\cap(\Ha_{\infty}\cup\overline{\Ha_{\infty}})$. The local period map $\widetilde{\Prd}\colon \F^m\longrightarrow \D\sqcup\overline{\D}$ has image contained in $\D\sqcup\overline{\D}-\Ha_s-\overline{\Ha_s}$.

\begin{prop}
\label{proposition: properties of local period map}
The local period map $\widetilde{\Prd}\colon \F^m\longrightarrow \D\sqcup\overline{\D}$ is an open embedding, with image either $\D-\Ha_s$ or $\D\sqcup\overline{\D}-\Ha_s-\overline{\Ha_s}$. In particular, $n^{\prime}=n$.
\end{prop}

\begin{proof}
We have a closed embedding $\pi\colon \D\sqcup \overline{\D}\hookrightarrow \widetilde{\D}$. There is a natural map $j\colon \F^m\longrightarrow \M^m$. Suppose $(X_1,\Phi_1),(X_2,\Phi_2)$ represent the same point in $\M^m$, then there exists a linear isomorphism $g\colon X_1\cong X_2$ such that
\begin{equation*}
g^*=\Phi_1^{-1}\circ \Phi_2\colon H^4(X_2,\ZZ)\longrightarrow H^4(X_1,\ZZ)
\end{equation*}
Since $\Phi_1,\Phi_2$ are compatible with the action of $A$ on $H^4(X_1,\ZZ)$, $H^4(X_2,\ZZ)$, so is $g^*$. Lemma \ref{lemma: isomotphism between automorphism groups} implies that $g$ is compatible with the actions of $A$ on $X_1,X_2$. Thus $(X_1,\Phi_1),(X_2,\Phi_2)$ represent the same point in $\F^m$. We showed the injectivity of $j$.

Combining with the following commutative diagram:
\begin{equation*}
\begin{tikzcd}
\F^m \arrow{r}{\widetilde{\Prd}}\arrow[hook]{d}{j} &\D\sqcup\overline{\D}\arrow[hook]{d}{\pi}\\
\M^m \arrow[hook]{r}{\widetilde{\Prd}} &\widetilde{\D}
\end{tikzcd}
\end{equation*}
we obtain the injectivity of $\widetilde{\Prd}\colon \F^m\longrightarrow \D\sqcup\overline{\D}$. In particular, $n\le n^{\prime}$.

Since the differential of $\widetilde{\Prd}\colon \M^m\longrightarrow \widetilde{\D}$ is injective everywhere, so is the differential of  $\widetilde{\Prd}\colon \F^m\longrightarrow \D\sqcup\overline{\D}$.

Take $(X,\Phi)\in \F^m$. Let $x=\Phi(H^{3,1}(X))\in \D\sqcup\overline{\D}$ and $y$ be any point in the component of $\D\sqcup\overline{\D}$ containing $x$. Since both $\D-\Ha_s$ and $\overline{\D}-\overline{\Ha_s}$ are connected, there exists a path
\begin{equation*}
\gamma\colon [0,1]\longrightarrow \D\sqcup\overline{\D}-\Ha_{s}-\overline{\Ha_{s}}
\end{equation*}
with $\gamma(0)=x$ and $\gamma(1)=y$. The path $\gamma$ has a unique lifting in $\M^m$. By proposition \ref{prop: strong global torelli for cubic fourfolds}, we can choose a family $\X\longrightarrow [0,1]$ of cubic fourfolds, with marking $\Phi$ of every fiber, such that $(\X_0,\Phi)=(X,\Phi)$ and $\Phi(H^{3,1}(\X_s))=\gamma(s)$, for all $s\in [0,1]$. Since $\gamma(s)\in\D\sqcup\overline{\D}$, the Hodge structure on $H^4(X_s,\ZZ)$ has action of $A$ induced by $\Phi$. By lemma \ref{lemma: isomotphism between automorphism groups}, there exists an action of $A$ on $\X_s$ for any $s\in [0,1]$, inducing compatible action on $H^4(\X_s,\ZZ)$. By lemma \ref{lemma: glue automorphisms}, actions of $A$ are of the same type $T$. Thus we obtain a lifting of $\gamma$ in $\F^m$, hence $y\in \widetilde{\Prd}(\F^m)$.

If $\widetilde{\Prd}(\F^m)\subset \D$, then $\widetilde{\Prd}(\F^m)=\D-\Ha_s$; otherwise $\widetilde{\Prd}(\F^m)$ intersects with both $\D$ and $\overline{\D}$, which implies that $\widetilde{\Prd}(\F^m)=\D\sqcup\overline{\D}-\Ha_s-\overline{\Ha}_s$.
\end{proof}

We introduce an involution on $\M^m$. Take any smooth cubic fourfold $X=Z(F)$, and a marking $\Phi\colon H^4(X,\ZZ)\longrightarrow \Lambda$. Let $X^{\prime}=Z(\overline{F})$. There exists a homeomorphism $\tau$ from $X$ to $X^{\prime}$ given by complex conjugation. Let $\iota$ be the involution on $\M^m$ sending $(X,\Phi)$ to $(X^{\prime}, \Phi\circ \tau^*)$. Consider smooth cubic fourfold $X=Z(F)$ such that $F$ has real coefficients. Then $\tau$ is a diffeomorphism of $X$, and $\tau^*$ sends $H^{3,1}(X)$ to $H^{1,3}(X)$. Therefore, choosing any marking $\Phi$ of $X$, the points $[(X,\Phi)]$ and $[(X,\Phi\circ \tau^*)]$ lie in different components of $\M^m$. This implies that the involution $\iota$ exchanges the two components of $\M^m$.

Next we give criterions on number of connected components of $\F^m$. For a symmetry type $T=[(A,\lambda)]$, we define the complex conjugate $\overline{T}$ of $T$ to be $[(\widetilde{A},\widetilde{\lambda})]$, where $\widetilde{A}$ is the complex conjugate of $A$, and $\widetilde{\lambda}(a)=\lambda(\overline{a})$ for all $a\in \widetilde{A}$. From definition, the involution $\iota$ exchanges the two spaces $\F_T^m$ and $\F_{\overline{T}}^m$.
\begin{prop}
\label{proposition: one or two components}
Given a symmetry type $T=[(A,\lambda)]$.
\begin{enumerate}[(i)]
\item If $\zeta$ is not real, then $\F^m$ is connected.
\item If $T=\overline{T}$, then $\F^m$ has two components.
\item If $T$ satisfies condition \ref{condition: smooth and identification of groups}, and $T\ne\overline{T}$, then $\F^m$ is connected.
\end{enumerate}
\end{prop}
\begin{proof}
Suppose $\zeta$ is not real, then $\widetilde{\Prd}(\F^m)$ lies in the ball associated to $(\Lambda_{\zeta},h)$. Thus $\F^m$ is connected.

Suppose $T=\overline{T}$, then $\F^m$ is preserved by $\iota$. Thus $\F^m$ has two components.

Suppose $\F^m$ has two components, then $\widetilde{\Prd}(\F^m)=\D\sqcup \overline{\D}-\Ha_s-\overline{\Ha_s}$. Thus $\F^m$ is preserved by $\iota$. Thus $\F_{\overline{T}}^m=\F_T^m$. This can not happen if $T$ satisfies condition \ref{condition: smooth and identification of groups} and $T\ne \overline{T}$. The third part follows.
\end{proof}

\subsection{Global Period Map}
\label{subsection: global period map}
In this section we are going to define the global period domain for symmetric cubic fourfolds of type $T$ as an arithmetic quotient of $\D$, and study the global period map.

Let $(d,k)=(3,4)$ and fix a symmetry type $T=[(A,\lambda)]$ satisfying condition \ref{condition: smooth}. Let $\Gamma=\{\rho\in \widehat{\Gamma}\big{|}\rho \overline{A}\rho^{-1}=\overline{A}\}$ be the normalizer of $\overline{A}$ in $\widehat{\Gamma}$. Take $\rho\in \widehat{\Gamma}$ and a point $x\in \Lambda_{\zeta}$. We claim that $\rho x\in \Lambda_{\zeta}$. In fact, take any $a\in A$, we have
\begin{equation*}
a\rho x=\rho \rho^{-1}a\rho x=\rho\zeta(\rho^{-1}a\rho)x=\zeta(\rho^{-1}a\rho)\rho x.
\end{equation*}
Since $\rho\in\widehat{\Gamma}$, we have $\rho [x]\in \widehat{\D}$. The two characters $\zeta$ and $\rho^{-1} \zeta\rho$ both give non-definite eigensubspaces of $\Lambda_{\CC}$. We conclude that $\zeta=\rho^{-1}\zeta\rho$, hence $\rho x\in \Lambda_{\zeta}$. This gives a natural action of $\Gamma$ on $\D$.

Let $N_A$ be the normalizer of $A$ in $\Aut((\Lambda_0)_{\QQ},\varphi)$, which is a reductive algebraic subgroup. The group $\Gamma$ is an arithmetic subgroup of $N_A$, see also appendix. The arithmetic quotient $\Gamma\bs \D$ is a quasi-projective variety thanks to the Baily-Borel compactification (see section \ref{subsection: functoriality of baily-borel compactification} in appendix). We denote $(\F^m)^1$ to be the connected component of $\F^m$ such that $\widetilde{\Prd}((\F^m)^1)=\D-\Ha_s$.
\begin{prop}
\label{proposition: global period map as isomorphism}
The local period map $\widetilde{\Prd}\colon (\F^m)^1\longrightarrow \D-\Ha_s$ descends to an algebraic isomorphism $\Prd\colon \F\cong \Gamma\bs (\D-\Ha_s)$.
\end{prop}
\begin{proof}
There are natural analytic morphisms from $\F^m$ to $\F$, and $\D-\Ha_s$ to $\Gamma\bs (\D-\Ha_s)$ respectively. We define the global period map $\Prd\colon \F\longrightarrow \Gamma\bs(\D-\Ha_s)$ as follows. Take $F\in \V^{sm}$. We choose a $T$-marking $\Phi$ of $X=Z(F)$, such that $\Phi(H^{3,1}(X))\in \D$ (this also means that $(F,\Phi)\in (\F^m)^1$). We define
\begin{equation*}
\Prd([F])=[\widetilde{\Prd}(X,\Phi)].
\end{equation*}

We show this map is well-defined. Take $F_1,F_2\in\V^{sm}$ with $T$-markings $\Phi_1,\Phi_2$ respectively. Suppose there exists $g\in N$, such that $g(F_1)=F_2$. We have an induced map
\begin{equation*}
g^*\colon H^4(Z(F_2),\ZZ)\longrightarrow H^4(Z(F_1),\ZZ).
\end{equation*}
Next we show $\rho=\Phi_1 g^* \Phi_2^{-1}\in\Gamma$. Denote $a^{\prime}=gag^{-1}$. Since $g\in N$, we have $a^{\prime}\in A$. we have the following commutative diagram:
\begin{equation*}
\begin{tikzcd}
\Lambda \arrow{r}{\Phi_2^{-1}}\arrow{d}{a^{\prime}} & H^4(Z(F_2),\ZZ)\arrow{r}{g^*}\arrow{d}{a^{\prime *}} &H^4(Z(F_1),\ZZ)\arrow{r}{\Phi_1}\arrow{d}{a^*}&\Lambda\arrow{d}{a}\\
\Lambda \arrow{r}{\Phi_2^{-1}} & H^4(Z(F_2),\ZZ)\arrow{r}{g^*} &H^4(Z(F_1),\ZZ)\arrow{r}{\Phi_1}&\Lambda
\end{tikzcd}
\end{equation*}
This implies that, as automorphisms of $\Lambda$, $a^{\prime}=\rho^{-1}a\rho$. Thus $\rho\in \Gamma$. We then have a well-defined analytic morphism $\Prd\colon \F\longrightarrow \Gamma\bs (\D-\Ha_s)$.

By definition we have the following commutative diagram:
\begin{equation}
\label{diagram: local and global period maps}
\begin{tikzcd}
(\F^m)^1\arrow{r}{\widetilde{\Prd}}\arrow{d}{j} &\D-\Ha_s \arrow{d}{\pi} \\
\F\arrow{r}{\Prd}  &\Gamma\bs(\D-\Ha_s).
\end{tikzcd}
\end{equation}
We next show $\Prd\colon \F\longrightarrow \Gamma\bs(\D-\Ha_s)$ is an isomorphism.

We first show injectivity. Suppose $(F_1,\Phi_1),(F_2,\Phi_2)\in \F^m$, with $\Phi_1(H^{3,1}(Z(F_1)))$ and $\Phi_2(H^{3,1}(Z(F_2)))$ representing the same point in $\Gamma\bs(\D-\Ha_s)$. Then there exists $\rho\in \Gamma$, such that $\rho\Phi_1(H^{3,1}(Z(F_1)))=\Phi_2(H^{3,1}(Z(F_2)))$. The map \begin{equation*}
\Phi_2^{-1}\rho\Phi_1\colon H^4(Z(F_1),\ZZ)\longrightarrow H^4(Z(F_2),\ZZ)
\end{equation*}
preserves the polarized Hodge structures. By lemma \ref{lemma: isomotphism between automorphism groups}, we have $g\in \SL(V)$, with $gF_2$ equals to $F_1$ after rescaling of $F_2$, and $g^*=\Phi_2^{-1}\rho\Phi_1$. For any $a\in A$, we have $a^*\colon H^4(Z(F_1),\ZZ)\longrightarrow H^4(Z(F_1),\ZZ)$. We have $g^{-1}ag$ acting on $Z(F_2)$, which induces:
\begin{equation*}
(g^{-1}ag)^*=g^*a^*g^{*-1}=(\Phi_2^{-1}\rho\Phi_1)(\Phi_1^{-1}a\Phi_1)(\Phi_1^{-1}\rho^{-1}\Phi_2)=\Phi_2^{-1}\rho a\rho^{-1}\Phi_2
\end{equation*}
Since $\rho\in \Gamma$, we have $\rho a\rho^{-1}\in A$. Again by lemma \ref{lemma: isomotphism between automorphism groups}, we have $g^{-1}ag\in A$. Since
\begin{equation*}
g^{-1}agF_2=g^{-1}aF_1=\lambda(a)g^{-1}F_1=\lambda(a)F_2,
\end{equation*}
we have $\lambda(g^{-1}ag)=\lambda(a)$. We conclude $g\in N$. Thus $\Prd$ is injective.

By proposition \ref{proposition: properties of local period map}, the composition of
\begin{equation*}
(\F^m)^1\longrightarrow \D-\Ha_s\longrightarrow \Gamma\bs(\D-\Ha_s)
\end{equation*}
is surjective. By commutativity of diagram \eqref{diagram: local and global period maps}, the composition of
\begin{equation*}
(\F^m)^1\longrightarrow \F\longrightarrow \Gamma\bs(\D-\Ha_s)
\end{equation*}
is also surjective, hence $\Prd\colon \F\longrightarrow \Gamma\bs(\D-\Ha_s)$ is surjective.

The algebraicity of $\Prd$ can be deduced from its extension to certain compactifications on both sides, see theorem \ref{theorem: two normalizations}. An alternative argument follows the proof of proposition 2.2.3 in \cite{hassett2000special} using Baily-Borel compactification and Borel extension theorem.
\end{proof}
\section{Compactifications}
\label{section: compactification}
In this section we are going to study the compactifications of both two sides of $\Prd\colon \F\longrightarrow \Gamma\bs(\D-\Ha_s)$. The essential ingredient is the identification of the $\GIT$-compactification of the moduli space of cubic fourfolds and the Looijenga compactification of the global period domain, proved by Looijenga \cite{looijenga2009period} and Laza \cite{laza2010moduli} independently. Depending on this, we will prove theorem \ref{theorem: main2}, and then deduce $(iii)$ of theorem \ref{theorem: main1}. In theorem \ref{theorem: criterion for bb}, we give a criterion when the Looijenga compactification is actually Baily-Borel compactification.

Let $(d,k)=(3,4)$. Recall that from theorem \ref{theorem: laza extension} we have isomorphism $\Prd\colon \M_1\cong \widehat{\Gamma}\bs (\widehat{\D}-\Ha_{\infty})$. From \cite{looijenga2009period} and \cite{laza2010moduli} we have:

\begin{thm}[Looijenga, Laza]
The period map $\Prd$ extends to $\Prd\colon \overline{\M}\longrightarrow \overline{\widehat{\Gamma}\bs \widehat{\D}}^{\Ha_{\infty}}$.
\end{thm}

Recall that $\Ha_*=\D\cap (\Ha_{\infty}\cup\overline{\Ha}_{\infty})$, which is a $\Gamma$-invariant hyperplane arrangement in $\D$. We have a morphism between locally symmetric varieties
\begin{equation*}
\Gamma\bs\D\longrightarrow \Aut(\Lambda,\eta)\bs\widetilde{\D}\cong \widehat{\Gamma}\bs\widehat{\D}.
\end{equation*}
We can construct the Looijenga compactification $\overline{\Gamma\bs \D}^{\Ha_*}$ of $\Gamma\bs(\D-\Ha_*)$ (see appendix \ref{section: appendix}). From theorem \ref{theorem: normalization looijenga compactification}, we have:
\begin{prop}
\label{prop: arithmetic quotient side normalization}
There exists finite morphism $\pi\colon \overline{\Gamma\bs\D}^{\Ha_*}\longrightarrow \overline{\widehat{\Gamma}\bs \widehat{\D}}^{\Ha_{\infty}}$. If $T$ satisfies condition \ref{condition: smooth and identification of groups}, then this morphism is a normalization of its image.
\end{prop}

We now state our main theorem:
\begin{thm}
\label{theorem: two normalizations}
The global period $\Prd\colon \F\cong \Gamma\bs(\D-\Ha_s)$ extends to an algebraic isomorphism $\Prd\colon \overline{\F}\cong \overline{\Gamma\bs \D}^{\Ha_*}$.
\end{thm}

We need the following fact in algebraic geometry. We give the proof for reader's convenience.
\begin{lem}
\label{lemma: uniqueness of semi-normalization}
Let $f_1\colon Z_1\longrightarrow Y$ and $f_2\colon Z_2\longrightarrow Y$ be finite morphisms between irreducible algebraic varieties. Suppose $Z_1,Z_2$ are normal. Moreover, there exists Zariski-open subset $U_i$ of $Z_i$, $i=1$ or $2$, with a biholomorphic map $g\colon U_1\longrightarrow U_2$, such that $f_1=f_2\circ g$. Then $g$ extends to an algebraic isomorphism $Z_1\longrightarrow Z_2$.
\end{lem}
\begin{proof}
Let $\CC(Z)$ be the field of rational functions on an irreducible algebraic variety $Z$, and $M(Z)$ the field of meromorphic functions. We claim $g^* \CC(Z_2)=\CC(Z_1)$. Let $x\in \CC(U_2)=\CC(Z_2)$. Since $\CC(U_2)$ is a finite extension of $\CC(Y)$, $g^* x$ is finite over $\CC(U_1)$. We can find a Zariski-open subset $U_1^{\circ}$ of $U_1$, with a Galois covering $\widetilde{U}\longrightarrow U_1^{\circ}$, such that $g^* x\in \CC(\widetilde{U})$. Since $g^* x\in M(U_1^{\circ})$, it is invariant under the action of Deck transformations. Thus $g^* x\in \CC(U_1^{\circ})=\CC(Z_1)$. The claim follows.

Without loss of generality, we assume $Y$ is affine. The coordinate ring $\CC[Z_i]$ is the integral closure of $\CC[Y]$ in $\CC(Z_i)$. So $g^* \CC[Z_2]=\CC[Z_1]$. Thus $g$ extends to an algebraic isomorphism $Z_1\cong Z_2$.
\end{proof}

\begin{proof}[Proof of theorem \ref{theorem: two normalizations}]
We have the following commutative diagram:
\begin{equation}
\label{diagram: two normalizations}
\begin{tikzcd}
\F\arrow{r}{\cong}\arrow[hook]{d} &\Gamma\bs(\D-\Ha_s)\arrow[hook]{d}\\
\overline{\F}\arrow{d}{j} & \overline{\Gamma\bs \D}^{\Ha_*}\arrow{d}{\pi}\\
\overline{\M}\arrow{r}{\Prd} &\overline{\widehat{\Gamma}\bs\widehat{\D}}^{\Ha_{\infty}}
\end{tikzcd}
\end{equation}
with both $j,\pi$ finite morphisms. Since $\F$ is Zariski-open in $\overline{\F}$, the image $j(\F)$ contains a Zariski-open subset of $j(\overline{\F})$. Thus $j(\overline{\F})$ is the closure of $j(\F)$ in $\overline{\M}$. The same argument shows that $\pi(\overline{\Gamma\bs\D}^{\Ha_*})$ is the closure of $\pi(\Gamma\bs(\D-\Ha_s))$ in $\overline{\widehat{\Gamma}\bs\widehat{\D}}^{\Ha_{\infty}}$. By commutativity of diagram \eqref{diagram: two normalizations}, the two images $j(\F)$ and $\pi(\Gamma\bs(\D-\Ha_s))$ are identified via $\Prd$, so are $j(\overline{\F})$ and $\pi(\overline{\Gamma\bs\D}^{\Ha_*})$. By proposition \ref{prop: git normalization}, proposition \ref{prop: arithmetic quotient side normalization} and lemma \ref{lemma: uniqueness of semi-normalization}, we have an identification between $\overline{\F}$ and $\overline{\Gamma\bs\D}^{\Ha_*}$ which extends $\Prd\colon \F\cong \Gamma\bs(\D-\Ha_s)$. This identification is the extended global period map $\Prd\colon \overline{\F}\cong \overline{\Gamma\bs\D}^{\Ha_*}$.
\end{proof}

The proof of the above theorem does not use algebraicity of $\Prd$. Actually, we can deduce algebraicity of $\Prd$ from theorem \ref{theorem: two normalizations}. At this point, we already finish the proof of part $(i),(ii)$ of theorem \ref{theorem: main1} and theorem \ref{theorem: main2}. In the rest of this section, we prove part $(iii)$ of theorem \ref{theorem: main1}.

Let $\V_1$ be the subset of $\V$ consisting of cubic forms of type $T$ defining cubic fourfolds with at worst $\ADE$-singularities. The points in $\V_1$ are stable with respect to the action of $\SL(V)$ on $\Sym^3(V^*)$, hence also stable with respect to the action of $N$ on $\V$. Define $\F_1=N\dbs \PP \V_1$ the moduli space of cubic fourfolds of type $T$ with at worst $\ADE$-singularities. We have:
\begin{prop}
\label{proposition: extension theorem for cubic fourfolds of type T with ADE singularities}
The period map $\Prd\colon \F\longrightarrow \Gamma\bs(\D-\Ha_s)$ extends to an algebraic isomorphism $\Prd\colon\F_1\cong \Gamma\bs(\D-\Ha_*)$.
\end{prop}
\begin{proof}
From definition we have $j(\F_1)=j(\overline{\F})\cap \M_1$ and $j^{-1}(j(\F_1))=\F_1$. From proposition \ref{prop: git normalization}, the morphism $j\colon\F_1\longrightarrow \M_1$ is finite. On the other hand, we have \begin{equation*}
\pi(\Gamma\bs(\D-\Ha_*))=\pi(\overline{\Gamma\bs\D}^{\Ha_*})\cap \widehat{\Gamma}\bs(\widehat{\D}-\Ha_{\infty})
\end{equation*}
and
\begin{equation*}
\pi^{-1}(\pi(\Gamma\bs(\D-\Ha_*)))=\Gamma\bs(\D-\Ha_*).
\end{equation*}
From proposition \ref{prop: arithmetic quotient side normalization}, the morphism $\pi\colon\Gamma\bs(\D-\Ha_*)\longrightarrow \widehat{\Gamma}\bs(\widehat{\D}-\Ha_{\infty})$ is finite. By theorem \ref{theorem: laza extension} and theorem \ref{theorem: two normalizations}, the two images $j(\F_1)$ and $\pi(\Gamma\bs(\D-\Ha_*))$ are identified via $\Prd$. By lemma \ref{lemma: uniqueness of semi-normalization}, we have algebraic isomorphism $\Prd\colon \F_1\cong \Gamma\bs(\D-\Ha_*)$.
\end{proof}

If the hyperplane arrangement $\Ha_*$ is empty, then the Looijenga compactification of $\Gamma\bs
\D$ is actually the Baily-Borel compactification. In the rest of this section, we give a criterion of emptiness of $\Ha_*$ from the perspective of $\GIT$. Following the notation in section 6 of \cite{laza2009moduli}, there is a rational curve $\chi$ parametrizing certain semi-stable cubic fourfolds, given by:
\begin{equation}
\label{equation: chi}
g_{a,b}(x_0,\dots,x_5)=
\left|\begin{array}{ccc}
x_0 & x_1 & x_2+2ax_5 \\
x_1 & x_2-ax_5 & x_3 \\
x_2+2ax_5 & x_3 & x_4
\end{array}\right|
+bx_5^3
\end{equation}
where $(a:b)\in WP(1:3)$. We denote the cubic fourfold corresponding to $(a:b)$ to be $X_{(a:b)}$. When $b=0$, the corresponding cubic fourfold $X_{(1:0)}$ is the determinantal cubic fourfold. The singular locus of $X_{(1:0)}$ is the image of the Veronese embedding $\PP V_3\hookrightarrow \PP\Sym^2(V_3)\cong \PP V$. Here $V_3$ is a three dimensional complex vector space with $\Sym^2(V_3)\cong V$. This induces a natural map from $\GL(V_3)$ to $\GL(V)$. Let $G_1$ be the intersection of $\SL(V)$ with the image of $\GL(V_3)\longrightarrow \GL(V)$.

For $b\ne 0$, the singular locus of the cubic fourfold $X_{(a:b)}$ is the image of $\PP V_2\hookrightarrow \PP(\Sym^4(V_2)\oplus \CC)\cong\PP V$. Here $V_2$ is a two dimensional complex vector space with $\Sym^4(V_2)\oplus \CC\cong V$. Let $\widetilde{G_2}$ be the subgroup of $\GL(V_2)\times \CC^*$ consisting of elements $(g,u)$ such that $(\det g)^2/u$ is a third root of unity. Let $G_2$ be the intersection of $\SL(V)$ with the image of the natural map $\widetilde{G_2}\longrightarrow \GL(V)$. The center of $\SL(V)$ is contained in both $G_1$ and $G_2$.

The automorphism group $\Aut(X_{(1:0)})$ is the image of $G_1$ in $\PSL(V)$, and this induces a character $\lambda_1$ of $G_1$. Explicitly, for $g\in \GL(V_3)$ representing an element in $G_1$, we have $\lambda_1(g)=\det(g)^4$. The automorphism group of $X_{(0,1)}$ is the image of $G_2$ in $\PSL(V)$, and this induces a character $\lambda_2$ of $G_2$ with $\lambda_2((g,u))=u^3$. We have the following criterion:
\begin{thm}
\label{theorem: criterion for bb}
The following three statements are equivalent:
\begin{enumerate}[(i)]
\item The hyperplane arrangement $\Ha_*$ is nonempty,
\item The space $\V_{\lambda}$ intersects with $\SL(V)\chi$,
\item The pair $(A,\lambda)$ factor through $(G_1,\lambda_1)$ or $(G_2,\lambda_2)$ defined as above.
\end{enumerate}
\end{thm}
\begin{proof}
The discussion above shows that $(ii)$ and $(iii)$ are equivalent.

We next show the first two statements are equivalent. Firstly we show that the image of $j\colon \overline{\F}\to\overline{\M}$ intersects with the image of $\chi$ in $\overline{\M}$ if and only if $(ii)$ holds. If $(ii)$ holds, the intersection point survives after taking $\GIT$ quotients since the $\SL(V)$ orbits of points in $\chi$ are closed. If $j(\overline{\F})$ intersects with the image of $\chi$ at $[F]$ in $\overline{M}$, then we take the representative $F$ in $\V_{\lambda}$ with closed $N$-orbit. According to the main theorem in \cite{luna1975Adh}, the $\SL(V)$-orbit of $F$ is also closed. So $F$ is contained in $\SL(V)\chi$.

Secondly we recall that the blow-up and blow-down construction in Looijenga compactification $\overline{\widehat{\Gamma}\bs\widehat{\D}}^{\Ha_{\infty}}$ gives a strata of projective line $\PP^1$ corresponding to $\chi$. We claim that $\Ha_*$ is nonempty if and only if the image of $\overline{\Gamma\bs\D}^{\Ha_*}$ intersects with the $\PP^1$. From the proof of functoriality of semi-toric compactification in appendix \ref{subsection: functoriality of semitoric blow up}, we know that $\D^\Sigma$ intersects with $\overline{\Ha}_\infty$ if and only if $\D$ intersects with $\Ha_{\infty}$. So the image of $\overline{\Gamma\bs\D}^{\Ha_*}$ intersects with the $\PP^1$ if and only if $\D$ intersects with $\Ha_\infty$. The equivalence of $(i)$ and $(ii)$ follows.
\end{proof}

We will apply this criterion to prime-order groups, see proposition \ref{proposition: whether bb}.

\section{Examples and Related Constructions}
\label{section: examples and conjectures}
Now given a symmetry type $T=[(A,\lambda)]$ for cubic fourfolds, we can obtain a global period map $\Prd\colon \overline{\F}\cong\overline{\Gamma\bs\D}^{\Ha_*}$. A closely related question is to classify automorphism groups of cubic fourfolds. There are 13 conjugacy classes of prime-order automorphisms of smooth cubic fourfolds (see \cite{gonzalez2011automorphisms}). For two of them, our main theorems recover some of the main results in \cite{allcock2002complex, allcock2011moduli}, \cite{looijenga2007period} and \cite{laza2017moduli}. We will discuss these examples in more details in sections \ref{subsection: classification prime order} and \ref{subsection: examples revisit}.

Cubic fourfolds have very close relation with hyper-K\"ahler manifolds, see \cite{beauville1985variety},\cite{hassett2000special}. We briefly recall the story. For a smooth cubic fourfold $X$, consider its Fano scheme of lines $F_1(X)$. This is a hyper-K\"ahler fourfold of $K3^{[2]}$ type. Automorphism group of a smooth cubic fourfold $X$ is naturally identified with the automorphism group of the polarized hyper-K\"ahler manifold $F_1(X)$ (the polarization is from Veronese embedding), see \cite{fu2016classification}.

The classification of automorphisms and automorphism groups of hyper-K\"ahler manifolds have appealed a lot of interests recently. A celebrated result of Mukai (\cite{mukai1988finite}) says that there are 11 maximal finite groups of symplectic automorphisms of $K3$ surfaces. More precisely, these 11 groups are exactly those maximal subgroups of the Mathieu group $M_{23}$ with at least 5 orbits in their induced action on $\{1,2,\dots,24\}$. The proof by Mukai was simplified by Kond\=o using Niemeier lattices (\cite{kondo1998niemeier}). It turns out that using Leech lattice instead of Niemeier lattices, one can obtain a more uniform treatment (see \cite{gaberdiel2012symmetries} and \cite{huybrechts2016derived}). About higher dimensional cases, there is a systematic study by Mongardi in his thesis (\cite{mongardi2012symplectic,mongardi2013symplectic,mongardi2016towards}).

In \cite{hohn2014finite}, H\"ohn and Mason classified all maximal symplectic automorphism groups of hyper-K\"ahler fourfolds of $K3^{[2]}$ type. Those groups are all subgroups of the Conway group (automorphism group of the Leech lattice quotient by its center).

Another closely related problem is to characterize the moduli spaces of symmetric or lattice-polarized hyper-K\"ahler manifolds. There are works \cite{dolgachev2007moduli} (section 11), \cite{artebani2011K3} (section 9), \cite{boissiere2015complex}, \cite{joumaah2016nonsymplectic}, \cite{camere2016lattice} (section 3), \cite{boissiere2016classification} (section 5) along this direction.
\subsection{Prime-order Automorphisms of Smooth Cubic Fourfolds}
\label{subsection: classification prime order}
The classification of prime-order automorphisms of smooth cubic fourfolds was given in \cite{gonzalez2011automorphisms} (theorem 3.8). For readers' convenience we present the result in this section. (There was a small mistake in \cite{gonzalez2011automorphisms}, theorem 3.8. The second example with $p=5$ contains only singular cubic fourfolds. This is pointed out in \cite{boissiere2016classification}, remark 6.3).

\begin{thm}[\cite{gonzalez2011automorphisms}]
Let $\omega$ be a prime $p$-th root of unity and $\rho=(m_0, \cdots, m_5)$ be the automorphism of $V\cong \CC^6$ given by $(x_0,\cdots, x_5)\mapsto (\omega^{m_0} x_0,\cdots, \omega^{m_5}x_5)$.
The list of smooth cubic polynomials $F$ preserved by the action under $\rho$ is as follows:
\begin{align*}
T_2^1\colon &\rho=(0,0,0,0,0,1), n=14, \\
&F=L_3(x_0,\cdots, x_4)+x_5^2L_1(x_0,\cdots, x_4).\\
T_2^2\colon &\rho=(0,0,0,0,1,1), n=12, \\
&F=L_3(x_0,\cdots, x_3)+x_4^2L_1(x_0,\cdots, x_3)+x_4x_5M_1(x_0,\cdots, x_3)+x_5^2N_1(x_0,\cdots, x_3).\\
T_2^3\colon &\rho=(0,0,0,1,1,1), n=10, \\
&F=L_3(x_0, x_1, x_2)+x_0L_2(x_3,x_4, x_5)+x_1M_2(x_3,x_4, x_5)+x_2N_2(x_3,x_4, x_5).\\
T_3^1\colon &\rho=(0,0,0,0,0,1), n=10, \\
&F=L_3(x_0,\cdots, x_4)+x_5^3.\\
T_3^2\colon &\rho=(0,0,0,0,1,1), n=4, \\
&F=L_3(x_0,\cdots, x_3)+M_3(x_4, x_5).\\
T_3^3\colon &\rho=(0,0,0,0,1,2), n=8, \\
&F=L_3(x_0,\cdots, x_3)+x_4^3+x_5^3+x_4x_5M_1(x_0,\cdots, x_3).\\
T_3^4\colon &\rho=(0,0,0,1,1,1), n=2, \\
&F=L_3(x_0, x_1, x_2)+M_3(x_3, x_4, x_5).\\
T_3^5\colon &\rho=(0,0,0,1,1,2), n=7,\\
&F=L_3(x_0, x_1, x_2)+M_3(x_3, x_4)+x_5^3+x_3x_5L_1(x_0,x_1, x_2)+x_4x_5M_1(x_0,x_1, x_2).\qquad\qquad\qquad\qquad\qquad\qquad\\
T_3^6\colon &\rho=(0,0,1,1,2,2), n=8, \\
&F=L_3(x_0, x_1)+M_3(x_2, x_3)+N_3(x_4, x_5)+\sum_{i=1,2; j=3,4; k=5,6}a_{ijk} x_ix_jx_k.\\
T_3^7\colon &\rho=(0,0,1,1,2,2), n=6, \\
&F=x_2L_2(x_0, x_1)+x_3M_2(x_0, x_1)+x_4^2L_1(x_0, x_1)+x_4x_5M_1(x_0,x_1)+x_5^2N_1(x_0,x_1)+x_4N_2(x_2,x_3)+x_5O_2(x_2,x_3).\\
T_5^1\colon &\rho=(0,0,1,2,3,4), n=4, \\
&F=L_3(x_0, x_1)+x_2x_5L_1(x_0, x_1)+x_3x_4M_1(x_0, x_1)+x_2^2x_4+x_2x_3^2+x_3x_5^2+x_4^2x_5.\\
T_7^1\colon &\rho=(1,2,3,4,5,6), n=2, \\
&F=x_0^2x_4+x_1^2x_2+x_0x_2^2+x_3^2x_5+x_3x_4^2+x_1x_5^2+ax_0x_1x_3+bx_2x_4x_5\\
T_{11}^1\colon &\rho=(0,1,3,4,5,9), n=0, \\
&F=x_0^3+x_1^2x_5+x_2^2x_4+x_2x_3^2+x_1x_4^2+x_3x_5^2.\\
\end{align*}
Here the lower index is the prime $p$, the polynomials $L_i,M_i,N_i$ are of degree $i$, and $n$ is the dimension of the corresponding $\GIT$-quotient.
\end{thm}

\begin{rmk}
This classification offers $13$ symmetry types with $\# \overline{A}$ a prime number $2, 3,5,7$ or $11$. Those symmetry types may not satisfy condition \ref{condition: smooth and identification of groups}.
\end{rmk}

By Griffiths residue calculus (\cite{griffiths1969periods}), for a smooth cubic fourfold $X=Z(F)$, the complex line $H^{3,1}(X)$ is generated by $\Res_X(\frac{\Omega}{F^2})$. Here $\Omega=\Sigma_{i=0}^5 (-1)^i x_i dx_1\wedge\dots\wedge\widehat{dx_i}\wedge\dots \wedge dx_5$. By direct calculation, we have:

\begin{prop}
\label{proposition: ball or type four prime order}
\begin{enumerate}[(i)]
\item For type $T=T_2^2, T_3^3, T_3^4, T_3^6, T_5^1, T_7^1, T_{11}^1$, we have $\zeta=1$.
\item For type $T=T_2^1, T_2^3$, we have $\zeta=-1$
\item For type $T=T_3^1, T_3^2, T_3^5, T_3^7$, we have $\zeta(\rho)$ equals to $\omega$ or $\overline{\omega}$.
\end{enumerate}
\end{prop}

We already proved that $\Prd(\F^m)$ is either $\D-\Ha_s$ or $\D\sqcup\overline{\D}-\Ha_s-\overline{\Ha_s}$. From proposition \ref{proposition: one or two components}, we have:
\begin{prop}
\label{proposition: image of local period map for cubic fourfolds with prime order automorphism}
\begin{enumerate}[(i)]
\item If $T=T_3^1, T_3^2, T_3^5, T_3^7$, then $\D$ is a complex hyperbolic ball and $\widetilde{\Prd}(\F^m)=\D-\Ha_s$.
\item If $T=T_2^1, T_2^2, T_2^3, T_3^3, T_3^4, T_3^6$ or $T_7^1$, then $\D$ is a type $\IV$ domain and $\widetilde{\Prd}(\F^m)=\D\sqcup\overline{\D}-\Ha_s-\overline{\Ha_s}$.
\end{enumerate}
\end{prop}

Now we apply theorem \ref{theorem: criterion for bb} for prime-order cases.
\begin{prop}
\label{proposition: whether bb}
For $T=T_2^1,T_2^3,T_3^2, T_3^3, T_3^4, T_3^7, T_{11}^1$, we obtain isomorphism between $\GIT$ compactification $\overline{\F}$ with Baily-Borel compactification $\overline{\Gamma\bs\D}^{bb}$. For $T=T_2^2, T_3^5, T_3^6, T_5^1, T_7^1$, we do not obtain Baily-Borel compactification.
\end{prop}
\begin{proof}
We will only do the calculation for $p=2$, the other cases are similar. If $(A,\lambda)$ factors through $(G_1, \lambda_1)$, then there exists $g\in \GL(V_3)$ with order $2$, such that the image of $g$ generates $\overline{A}$. We can choose basis of $V_3$, such that the matrix corresponds to $g$ is $\diag (1,-1, -1)$. The image of $g$ in $GL(V)$ is $\diag(1,1,1,1,-1,-1)$. If $(A,\lambda)$ factors through $(G_1, \lambda_1)$, then we can choose $(g, u)\in \GL(V_2)\otimes \CC^*$ such that $g^2=id$ and $(\det g^2)/u$ is a third root of unity. Under suitable basis, we have $g=\diag(1,-1)$. Then the image of $(g,u)$ in $\SL(V)$ is $\diag(1,1, -1,-1,-1,-1)$. In both two cases, the characters $\lambda_1$ and $\lambda_2$ are trivial. By theorem \ref{theorem: criterion for bb}, the symmetry type $T_2^2$ does not give Baily-Borel compactification and $T_2^1, T_2^3$ give Baily-Borel compactifications.
\end{proof}
\subsection{Examples revisit}
\label{subsection: examples revisit}
Take $T=T_3^1$, then $T=[(\overline{A}=\mu_3,\lambda=1)]$ satisfies condition \ref{condition: smooth and identification of groups}. The space $\F$ can be identified with the moduli space of smooth cubic threefolds. The local period domain $\D$ is a complex hyperbolic ball of dimension $10$ with an action of an arithmetic group $\Gamma$. Then theorem \ref{theorem: main1} and theorem \ref{theorem: main2} recover the main results in \cite{looijenga2007period} and \cite{allcock2011moduli}. By proposition \ref{proposition: whether bb}, the hyperplane arrangement $\Ha_*$ is nonempty. Actually, from \cite{looijenga2007period} and \cite{allcock2011moduli}, the quotients $\Gamma\bs\Ha_s$ has two irreducible components, and $\Gamma\bs\Ha_*$ is irreducible.

Take $T=T_2^1$, then $T=[(\overline{A}=\mu_2,\lambda=1)]$ satisfies condition \ref{condition: smooth and identification of groups}. In this case, the moduli space $\F$ turns out to be the moduli space of pairs consisting of a cubic threefold and a hyperplane section. This is recently studied in \cite{laza2017moduli}. Denote $\mathcal{W}_1=H^0(\PP^4, \mathcal{O}(3))$ the space of cubic forms in $x_0,\dots,x_4$ and $\mathcal{W}_2=H^0(\PP^4, \mathcal{O}(1))$ to be the space of linear forms in $x_0,\dots,x_4$. We have an identification $\mathcal{W}_1\oplus\mathcal{W}_2\cong\V$ sending $(L_3,L_1)$ to $L_3+x_5^2 L_1$. In their paper \cite{laza2017moduli}, the authors defined $\F$ to be a $\GIT$-quotient of $(\PP\mathcal{W}_1\times \PP\mathcal{W}_2, \mathcal{O}(3)\boxtimes \mathcal{O}(1))$ by $\SL(5,\CC)$.  Direct calculation shows that $N=C=\SL(5,\CC)\times Z\subset\SL(V)$, where $Z=\{diag(u,u,u,u,u,u^{-5})\big{|}u\in\CC^{\times}\}$ is the center. The following proposition gives the relation of our constructions with that in \cite{laza2017moduli}:
\begin{prop}
We have identification between polarized projective varieties:
\begin{equation*}
Z\dbs (\PP\V,\mathcal{O}(1))\cong (\PP\mathcal{W}_1\times \PP\mathcal{W}_2, \mathcal{O}(3)\boxtimes \mathcal{O}(1)).
\end{equation*}
\end{prop}
\begin{proof}
It is equivalent to show
\begin{equation*}
\bigoplus_k (H^0(\PP\V,\mathcal{O}(k)))^Z\cong \bigoplus_k H^0(\PP\mathcal{W}_1\times \PP\mathcal{W}_2, \mathcal{O}(3k)\boxtimes \mathcal{O}(k))
\end{equation*}
as graded algebras. The action of $Z$ on $\mathcal{W}_1$ has weight $3$, and on $\mathcal{W}_2$ weight $-9$.

We have the direct sum decomposition
\begin{equation*}
\Sym^{m}(\V^*)=\bigoplus_{k+l=m} \Sym^{k}{\mathcal{W}_1^*}\otimes \Sym^{l}{\mathcal{W}_2^*}
\end{equation*}
with $Z$-action of weight $-3k+9l$. The weight zero part has $k=3l$ and $m=4l$. So we obtain identification of the two polarized varieties.
\end{proof}
Moreover, by proposition \ref{proposition: whether bb}, the hyperplane arrangement $\Ha_*$ is empty in this case, and we obtain identification between $\overline{\F}$ and Baily-Borel compactification $\overline{\Gamma\bs\D}^{bb}$. This recovers the main result in \cite{laza2017moduli}.


\begin{appendix}
\section{Locally Symmetric Varieties and Looijenga compactifications}
\label{section: appendix}
It is well-known that the normalization of each stratum in the orbifold loci of a locally Hermitian symmetric variety is still a locally Hermitian symmetric variety. For reader's convenience, we include a discussion of this fact in section \ref{subsection: orbifold locus of locally symmetric varieties}. In the rest of the appendix, we prove that similar result holds for Looijenga compactifications.

We will recall the construction of Looijenga compactifications of arithmetic quotients $\XX$ of complex hyperbolic balls or type $\IV$ domains. There are two steps. The first is the semi-toric blowup $\overline{\XX}^\Sigma$, which is an intermediate compactification of arithmetric quotient $\XX$ sitting between Baily-Borel and toroidal compactifications. We will recall the geometric construction of Baily-Borel compactifications of complex hyperbolic balls and type $\IV$ domains in section \ref{subsection: functoriality of baily-borel compactification}, and recall the semi-toric blow-up construction in section \ref{subsection: functoriality of semitoric blow up}. The second step is successive blow-up constructions along the hyperplane arrangement in $\overline{\XX}^\Sigma$ and blow-down construction of certain induced strata (We will sketch this in section \ref{subsection: main theorem in appendix}).

\subsection{Orbifold Loci of Locally Symmetric Varieties}
\label{subsection: orbifold locus of locally symmetric varieties}
In this section we show the normalization of an orbifold stratum of locally Hermitian symmetric variety is again locally Hermitian symmetric variety.

Let $G$ be a real reductive algebraic group with compact center. Let $K$ be a maximal compact subgroup of $G$. Let $\D=G/K$ be the corresponding symmetric space. Assume $\D$ is Hermitian symmetric and $G$ has a $\QQ$-structure. Let $\Gamma\subset G(\QQ)$ be an arithmetic subgroup. For simplicity, we assume the action of $\Gamma$ on $\D$ is faithful. Denote $\XX=\Gamma\bs \D$ to be the arithmetic quotient. This is naturally a quasi-projective variety due to Baily-Borel compactification (see \cite{borel1966}). There is a natural orbifold sturcture on $\XX$. We consider the orbifold locus indexed by certain finite subgroup $A\subset\Gamma$. More precisely, we take $A\subset \Gamma$ fixing some point $x\in \D$. Without loss of generality, we assume $K$ to be the stabilizer of $x\in \D$ under the action of $G$. Then $A\subset K$. Denote $G_A, K_A$ and $\Gamma_A$ to be the corresponding normalizers of $A$ in $G,K$ and $\Gamma$ respectively. Then $G_A$ is again a real reductive algebraic group with compact center and $K_A$ is a maximal compact subgroup (see \cite{looijenga2016moduli}, page 37-38). There is a natural holomorphic embedding
\begin{equation*}
G_A/K_A\hookrightarrow \D=G/K.
\end{equation*}
Define $\D_A\coloneqq G_A/K_A$. This is a Hermitian symmetric subspace of $\D$. We have the following proposition:
\begin{prop}
\label{proposition: normalization smooth case}
The group $\Gamma_A$ is an arithmetic subgroup in $G_A(\QQ)$ and the map $\pi\colon \Gamma_A\bs\D_A \longrightarrow \Gamma\bs\D$ is finite. Furthermore, if $A$ is the stabilizer of $x$ under the action of $\Gamma$, then this map gives a normalization of its image.
\end{prop}

\begin{proof}
Due to the extension theorem of Baily-Borel compactifications (see theorem 2 in \cite{kobayashi1972}), the map $\pi$ is algebraic and proper. We show $\pi$ is finite. It suffices to show $\pi$ is quasi-finite, namely, having finite fibers. Take any $y\in \D_A$. Suppose we have a point $y^{\prime}=\rho y$ for $\rho\in \Gamma$. Then $\rho^{-1} A\rho$ is contained in the stabilizer group of $y$. Actually, the $\Gamma_A$-orbits of such points $y^{\prime}$ are one-to-one corresponding to subgroups with form $\rho^{-1}A\rho$ in the stabilizer group of $y$, hence finitely many.

If $A$ is the stabilizer group of $x$, a generic point in $\XX_A\coloneqq \Gamma_A\bs\D_A$ also has $A$ as stabilizer group. We first show that $\pi$ is generically injective in this case. Take generically $x_1,x_2\in \D_A$, and assume they $[x_1]=[x_2]$ in $\Gamma\bs \D$. Then there exists $\rho\in \Gamma$ such that $\rho x_1=x_2$. Since both $x_1,x_2$ have stabilizer group $A$, we have $\rho A\rho^{-1}=A$, hence $\rho\in \Gamma_A$. This implies that $[x_1]=[x_2]$ in $\Gamma_A\bs \D_A$. We have $\pi$ a finite and birational morphism from a normal variety to its image, hence a normalization of its image.
\end{proof}

\begin{rmk}
The same construction also works for any finite volume locally Hermitian symmetric varieties. The difference from the arithmetic case is that $\Gamma_A$ is not automatically a lattice. We need to use the compactification in finite volume case (see theorem 1 in \cite{mok1989compactifying}) to show that the orbifold locus also admits a compactification, which implies the finiteness of the volume by Yau's Schwarz lemma (\cite{yau1978schwarz}).
\end{rmk}

\subsection{Orbifold Loci of Ball and Type $\IV$ Quotients}
\label{subsection: orbifold loci of ball/typefour quotients}
We now focus on arithmetic quotients of balls and type $\IV$ domains.

We fix the notation that will be used in the rest of the appendix. Let $(V_{\QQ}, \varphi)$ be a vector space over $\QQ$ with nondegenerate rational bilinear form $\varphi$ of signature $(2,N)$. Let $V=V_{\QQ}\otimes \CC$. Notice that here $V_{\QQ}$ is not necessarily the middle cohomology of cubic fourfold. Similar as section \ref{section: review of theory of cubic fourfolds}, the type $\IV$ domain $\widehat{\D}$ associated to $(V_{\QQ}, \varphi)$ is a component of
\begin{equation*}
\widehat{\D}\sqcup \overline{\widehat{\D}}=\PP\{x\in V\big{|}\varphi(x,x)=0, \varphi(x, \overline{x})>0\}.
\end{equation*}
Denote by $\widehat{G}$ the subgroup of $\Aut(\varphi)(\RR)$ (of index $2$) respecting the component $\widehat{\D}$. Let $\widehat{\Gamma}\subset \widehat{G}$ be an arithmetic subgroup. The corresponding locally Hermitian symmetric variety is $\widehat{\XX}=\widehat{\Gamma}\bs \widehat{\D}$. Let $A$ be a finite subgroup of $\widehat{\Gamma}$. Let $\zeta$ be a character of $A$, such that there exists $x\in V$ with $\varphi(x,x)=0$ and $\varphi(x,\overline{x})>0$, and $a(x)=\zeta(a)x$ for all $a\in A$. Denote $V_{\zeta}$ to be the $\zeta$-subspace of $V$. Then there is a natural Hermitian form $h$ on $V_{\zeta}$ defined by $h(x,y)=\varphi(x,\overline{y})$. If $\zeta=\overline{\zeta}$, this Hermitian form has signature $(2,n)$ and we obtain a type $\IV$ subdomain $\D$ of $\widehat{\D}$. Otherwise the signature is $(1,n)$ and we obtain a complex hyperbolic ball $\B$ inside $\widehat{\D}$. Indeed, let \begin{equation*}
G\coloneqq \{g\in\widehat{G})| gAg^{-1}=A\}
\end{equation*}
be an algebraic subgroup over $\QQ$. The fixed locus of $A$ in $\D$ is $G(\RR)/K$, where $K$ is maximal compact subgroup of $G(\RR)$. Denote $\Gamma=\{\rho\in\widehat{\Gamma}\big{|}\rho^{-1}A\rho=A\}$. The same as section \ref{section: local period map}, we have $\Gamma$ an arithmetic subgroup of $G(\QQ)$ acting on $\B$ or $\D$. Then we have a natural map $\Gamma\bs \D\longrightarrow \widehat{\Gamma}\bs\widehat{\D}$ or $\Gamma\bs \B\longrightarrow \widehat{\Gamma}\bs\widehat{\D}$. We consider the following condition:
\begin{cond}
\label{condition: arithmetic quotient part}
The group $A$ is the stabilizer of a generic point of $\D$ or $\B$.
\end{cond}
If $A$ satisfies this condition, proposition \ref{proposition: normalization smooth case} implies that the morphism $\pi\colon \Gamma\bs\B\longrightarrow \widehat{\Gamma}\bs\widehat{\D}$ or $\pi\colon \Gamma\bs\D\longrightarrow \widehat{\Gamma}\bs\widehat{\D}$ is the normalization of its image.

We will consider a larger set of type $\IV$ subdomains. Take $W_{\QQ}$ to be a $\QQ$-subspace of $V_{\QQ}$ with signature $(2,n)$, we have the associated type $\IV$ subdomain $\D$ inside $\widehat{\D}$ with the action of an arithmetic group $\Gamma_W=\{\rho\in\widehat{\Gamma}\big{|}\rho(W)=W\}$. Take $V_{\ZZ}$ to be an integral structure on $V_{\QQ}$ such that $\Gamma\subset \Aut(V_{\ZZ})$ has finite index. Denote $W_{\ZZ}\coloneqq W_{\QQ}\cap V_{\ZZ}$. For $x\in \D$, define $\Pic(x)\coloneqq V_x^{1,1}\cap V_\ZZ$ to be the Picard lattice of $x$ where $x$ is viewed as a weight two Hodge structure on $V_{\ZZ}$. Then for generic $x\in \D$, we have $\Pic(x)=W_{\ZZ}^{\perp}$.

We have the following lemma:
\begin{lem}
For $A$ satisfying condition \ref{condition: arithmetic quotient part} and $W=V_{\zeta}$, we have $\Gamma_A=\Gamma_W$.
\end{lem}
\begin{proof}
It is straightforward that  $\Gamma_A\subset\Gamma_W$, and they both act on $\D$. Take any $\rho\in \Gamma_W$ and a generic point $x$ in $\D$. Then $A$ is contained in the stabilizer group of $\rho x$. Thus both $A$ and $\rho^{-1} A\rho$ are contained in the stabilizer group of $x$. Since $x$ is generic, we have $\rho^{-1}A \rho=A$ by condition \ref{condition: arithmetic quotient part}. So $\rho\in \Gamma_A$. We showed that $\Gamma_W\subset\Gamma_A$.
\end{proof}

With this lemma, we will simply denote $\Gamma$ to be the arithmetic group acting on $\D$. We have:
\begin{prop}
For any $\QQ$-subspace $W_{\QQ}$ (of $V_{\QQ}$) with signature $(2,n)$, we have a morphism $\pi\colon\Gamma\bs\D\longrightarrow \widehat{\Gamma}\bs\widehat{\D}$, which is the normalization of its image.
\end{prop}
\begin{proof}
Properness is by \cite{kobayashi1972}. Take a generic point $x$ in $\D$. Suppose $\rho\in \widehat{\Gamma}$ sends $x$ to $\rho x\in \D$. The Picard lattice $\Pic(\rho x)$ of $\rho x$ contains $W_{\ZZ}^{\perp}$, hence $\rho^{-1}(W_{\ZZ}^{\perp})\subset \Pic(x)$. Since $x$ is generic, we have $\Pic(x)=W_{\ZZ}^{\perp}$. This implies that $\rho(W_{\ZZ}^{\perp})=W_{\ZZ}^{\perp}$, hence $\rho(W)=W$. Thus $\rho\in \Gamma_W$.

Finally, we show finiteness. Take a point $x\in \D$. For any $\rho \in \widehat{\Gamma}$, we have $\rho^{-1}(W_{\ZZ}^{\perp})$ contained in the Picard lattice $\Pic(x)$. The set $\widehat{\Gamma}x$ is a disjoint union of some $\Gamma$-orbits, each of which corresponds to the image of certain primitive embedding of $W_{\ZZ}^{\perp}$ into $\Pic(x)$. The orthogonal complement of $W_{\ZZ}^{\perp}$ in $\Pic(x)$ is positive definite with discriminant at most $\mathrm{det}(W_{\ZZ}^{\perp})\mathrm{det}(\Pic(x))$. By reduction theory of lattice, there are finitely many such primitive embeddings.
\end{proof}

\subsection{Functoriality of Baily-Borel Compactification}
\label{subsection: functoriality of baily-borel compactification}
In this section we recall Baily-Borel compactifications of arithmetic quotients of complex hyperbolic balls or type $\IV$ domains. See \cite{borel1966} and \cite{looijenga2003ball,looijenga2003typefour}.

We deal with type $\IV$ domain $\widehat{\D}$ first. The boundary components of Baily-Borel compactifications corresponds to $\QQ$-isotropic planes $J$ or $\QQ$-isotropic lines $I$. Let
\begin{equation*}
\pi_{J^\perp}\colon \PP(V)-\PP(J^\perp)\longrightarrow \PP(V/J^\perp)
\end{equation*}
and
\begin{equation*}
\pi_{I^\perp}\colon \PP(V)-\PP(I^\perp)\longrightarrow \PP(V/I^\perp)
\end{equation*}
be the natural projections. The image $\pi_{J^\perp}\widehat{\D}$ is isomorphic to upper half plane. The image $\pi_{I^\perp}\widehat{\D}$ is a point. Adding rational boundary components, we have
\begin{equation*}
\widehat{\D}^{bb}\coloneqq \widehat{\D}\sqcup\coprod_J \pi_{J^\perp}\widehat{\D}\sqcup\coprod _I\pi_{I^\perp}\widehat{\D}
\end{equation*}
with suitable topology and ringed space structure. The Baily-Borel compactification is the quotient $\Gamma\bs \widehat{\D}^{bb}$ as a projective variety.

 Given $W_{\QQ}\subset V_{\QQ}$ with signature $(2,n)$. Let $\D$ be the corresponding type $\IV$ domain. We have a natural map from $\D$ to $\widehat{\D}$, inducing $\Gamma\bs \D\longrightarrow \widehat{\Gamma}\bs\widehat{\D}$. According to theorem 2 in \cite{kobayashi1972}, this holomorphic map can be extended to Baily-Borel compactifications, sending boundary components into boundary components.

\begin{prop}[type $\IV$ to type $\IV$]
There is a natural finite extension of $\pi\colon\Gamma\bs \D\longrightarrow \widehat{\Gamma}\bs\widehat{\D}$ to Baily-Borel compactifications
\begin{equation*}
\pi\colon \overline{\Gamma\bs\D}^{bb}\longrightarrow \overline{\widehat{\Gamma}\bs\widehat{\D}}^{bb}.
\end{equation*}
If $A$ satisfies condition \ref{condition: arithmetic quotient part}, the map is a normalization of its image.
\end{prop}
\begin{proof}
Let $W\coloneqq V_\zeta$ in this proof. The boundary components of $\D^{bb}$ correspond to rational isotropic planes $J$ and rational isotropic lines $I$ in $W$. From the natural embedding $W\hookrightarrow V$, they also have associated boundary components in $\widehat{\D}^{bb}$. Under the following natural commutative diagram
\begin{equation*}
\begin{tikzcd}
\PP(W)-\PP(J^{\perp}) \arrow{r}{\pi_{J^{\perp}}}\arrow{d}  &\PP(W/J^{\perp})\arrow{d}\\
\PP(V)-\PP(J^{\perp}) \arrow{r}{\pi_{J^{\perp}}}  &\PP(V/J^{\perp})
\end{tikzcd}
\end{equation*}
we have isomorphisms $\pi_{J^\perp}\D\longrightarrow \pi_{J^\perp}\widehat{\D}$, and similar maps $\pi_{I^\perp}\D\longrightarrow \pi_{I^\perp}\widehat{\D}$, which induce an extension $\D^{bb}\longrightarrow \widehat{\D}^{bb}$ equivariant under the action of $\Gamma\longrightarrow \widehat{\Gamma}$. After taking quotients, we have an extension map $\XX^{bb}\longrightarrow \widehat{\XX}^{bb}$. By proposition \ref{proposition: normalization smooth case}, this map is generically injective and it is finite over $\widehat{\Gamma}\bs\widehat{\D}$. Let $\Gamma_J$ be the stabilizer of $J$ under the action of $\Gamma$. The projection of $\Gamma$ in $\GL(J)$ (or equivalently $\GL(V/J^\perp)$) is arithmetic. The boundary component corresponding to $J$ is the quotient of $\pi_{J^\perp}\widehat{\D}$ by $\Gamma_J$, hence a modular curve. The restriction to the boundary component corresponding to each $J$ is a non-constant map between modular curves, hence finite. The restriction to boundary components corresponding to each $I$ is automatically finite. So we have an algebraic finite morphism between normal varieties $\XX^{bb}\longrightarrow \widehat{\XX}^{bb}$. If $A$ satisfies condition \ref{condition: arithmetic quotient part}, then this morphism is generically injective by proposition \ref{proposition: normalization smooth case}, hence a normalization of its image.
\end{proof}

We recall the Baily-Borel compactification of Ball quotient. Let $K$ be a CM field and $W_K$ a finite dimensional vector space over $K$ with
\begin{equation*}
h_K\colon W_K\times W_K\longrightarrow K
\end{equation*}
a $K$-valued Hermitian form. For each embedding $\iota\colon K\hookrightarrow \CC$, we define $W_{\iota}\coloneqq W_K\otimes_{\iota}\CC$, then we have a Hermitian form $h_{\iota}\colon W_{\iota}\times W_{\iota}\longrightarrow \CC$. Assume the form $h_{\iota}$ has signature $(1,n)$ under embedding $\iota=\iota_1$ or $\overline{\iota}_1$, and is definite otherwise. The complex ball $\B$ is defined to be the set of positive lines in $W_{\iota_1}$. The boundary components of Baily-Borel compactification correspond to $K$-isotropic lines $I$ in $W_K$ and we denote $\B^{bb}\coloneqq \B\sqcup\coprod_I \pi_{I^\perp}\B$. When the totally real part of $K$ is not $\QQ$, there exists complex embedding $\iota$ such that $(W_{\iota}, h_{\iota})$ is definite, which implies that any isotropic vector in $W_K$ must be zero. Thus in this case the boundary set is empty.

Now consider the action of $A$ on $V$ with $\zeta\ne\overline{\zeta}$. Let $K$ be the cyclotomic field generated by $\zeta(A)$. Take $W_K$ to be the $\zeta$-eigenspace of $V_K\coloneqq V_{\QQ}\otimes K$ under the action of $A$.
\begin{lem}
\label{lemma: W isotropic}
The $K$-vector space $W_K$ is isotropic under $\varphi$.
\end{lem}
\begin{proof}
Take any $x,y\in W_K$, we need to show $\varphi(x,y)=0$. Take $a\in A$ such that $\zeta(a)$ is not real. Then $\zeta(a)^2\ne 1$. By
\begin{equation*}
\varphi(x,y)=\varphi(ax,ay)=\varphi(\zeta(a)x,\zeta(a)y)=\zeta(a)^2\varphi(x,y),
\end{equation*}
we have $\varphi(x,y)=0$.
\end{proof}
There is natural Hermitian form $h$ of signature $(1,n)$ on $W_K$, given by $h(x,y)=\varphi(x,\overline{y})$ for all $x,y\in W_K$. The Galois conjugates of $K$ define eigenspaces of $V$ under the action of $A$. The sum of all those eigenspaces is a subspace of $V$ defined over $\QQ$. Then we have the ball $\B$ consisting of positive lines in $W$ and we denote $\overline{(\Gamma\bs\B)}^{bb}\coloneqq \Gamma\bs \B^{bb}$ the Baily-Borel compactification of $\XX=\Gamma\bs \B$.
\begin{prop}[ball to type $\IV$]
There is a natural finite extension of $\pi\colon\Gamma\bs \B\longrightarrow \widehat{\Gamma}\bs\widehat{\D}$ to Baily-Borel compactifications
\begin{equation*}
\pi\colon \overline{\Gamma\bs\B}^{bb}\longrightarrow \overline{\widehat{\Gamma}\bs\widehat{\D}}^{bb}.
\end{equation*}
If $A$ satisfies condition \ref{condition: arithmetic quotient part}, the map is a normalization of its image.
\end{prop}

\begin{proof}
Similar as the proof for type $\IV$ case, we need to identify the boundary components on both sides. The ball and its boundaries are defined as above by $W_K$. If $K$ is not a quadratic extension of $\QQ$, then the boundary set is empty, hence $\Gamma\bs\B$ is already compact. If $K$ is, then each $K$-isotropic line $I$ together with its complex conjugate $\overline{I}$ defines a rational isotropic plane in $V_\QQ$. So there is a natural extension map $\B^{bb}\longrightarrow \widehat{\D}^{bb}$ which is equivariant under the action of $\Gamma\longrightarrow \widehat{\Gamma}$. After taking quotient on both sides, we have a finite algebraic map $\pi\colon\XX^{bb}\longrightarrow \widehat{\XX}^{bb}$. If $A$ satisfies condition \ref{condition: arithmetic quotient part}, then this morphism is generically injective by proposition \ref{proposition: normalization smooth case}, hence a normalization of its image.
\end{proof}

\begin{rmk}
Similar construction of ball quotients appears in the arithmetic examples of Deligne-Mostow theory, see \cite{deligne1986monodromy} and \cite{looijenga2007uniformization}. In both constructions, if the cyclotomic field generated by the corresponding characters is not $\QQ(\sqrt{-1})$ or $\QQ(\sqrt{-3})$, then the Baily-Borel compactification is compact.
\end{rmk}

\subsection{Functoriality of Semi-toric Compactifications}
\label{subsection: functoriality of semitoric blow up}
We first briefly sketch the semi-toric blow-up constructions of complex hyperbolic balls and type $\IV$ domains with respect to certain hyperplane arrangements. See \cite{looijenga2003ball, looijenga2003typefour}. Semi-toric compactification with respect to a hyperplane arrangement is the minimal blowup of certain boundary components in Baily-Borel compactification, such that the closure of every hypersurface is Cartier at the boundary.

Let $\widehat{\Ha}$ be a hyperplane arrangement on $\widehat{\D}$ defined by a set of negative vectors $v\in V_\QQ$, which form finitely many orbits under the action of $\widehat{\Gamma}$.  We recall some definitions and notation in \cite{looijenga2003typefour}. Each rational isotropic line $I$ in $V_\QQ$ realizes $\widehat{\D}$ as a tube domain, with real cone denoted by
\begin{equation*}
C_I\subset (I^{\perp}/I\otimes I)(\RR).
\end{equation*}
Each rational isotropic plane $J$ determines a half line on the boundaries of the $C_I$ for any $I\subset J$. The union of these cones is called the conical locus of $\widehat{\D}$. Let $C_{I,+}$ be the convex hull of $\overline{C}_I\cap (I^{\perp}/I\otimes I )(\QQ)$, which is the union of $C_{I}$ with rational isotropic half lines corresponding to $J$ containing $I$. The hyperplane arrangement $\widehat{\Ha}$ determines an admissible decomposition $\Sigma(\widehat{\Ha})$ of the conical locus. More precisely, it is a $\Gamma$-invariant choice of locally rational cone decomposition of $C_{I,+}$ such that the support for isotropic half line corresponding to $J$ is independent of those $I\subset J$. See section 6 of \cite{looijenga2003typefour} for details. For each member $\sigma$ of $\Sigma(\widehat{\Ha})$ contained in $C_{I, +}$, we define a corresponding vector subspace $V_{\sigma}$ of $V$ as follows. When $\sigma$ is the half line corresponding to an isotropic plane $J$, then
\begin{equation*}
V_\sigma\coloneqq (\underset{J\subset H}{\bigcap} H)\cap J^{\perp}.
\end{equation*}
Otherwise $V_\sigma$ is the span of $\sigma$ in $I^{\perp}$, which is also the intersection among $I^\perp$ and those $H\in \widehat{\Ha}$ containing $I$. Here we identify $H\subset V$ with $H\in \widehat{\Ha}$. We have a projection $\pi_{\sigma}\colon \D\longrightarrow \PP(V/V_{\sigma})$. The semi-toric compactifications is denoted by $\overline{\XX}^\Sigma=\Gamma\bs \D^{\Sigma}$. Here $\D^{\Sigma}\coloneqq\coprod_{\sigma\in \Sigma} \pi_{\sigma}\widehat{\D}$. The space $\overline{\XX}^{\Sigma}$ has a natural map to $\widehat{\XX}^{bb}$ respecting the stratifications. We have two different types of boundary components. One is finite quotient of abelian torsor over the modular curve $\widehat{\Gamma}_J\bs \pi_{J^\perp}\widehat{D}$. The abelian torsor is modeled over vector group $J^\perp/V_\sigma$ quotient by a lattice. The other is algebraic torus torsor over a point $\pi_{I^\perp}\widehat{D}$, which is the boundary stratum in the quotient of an infinite-type toric variety induced by the cone decomposition of $C_{I,+}$. In particular, each cone of codimension $k$ corresponds to algebraic torus torsor of dimension $k$.

Given $W_{\QQ}\subset V_{\QQ}$ a sublattice of signature $(2,n)$, with $\D$ the associated type $\IV$ domain. We have the intersection $\Ha\coloneqq \D\cap \widehat{\Ha} $ a $\Gamma$-invariant hyperplane arrangement in $\D$. We also have the semi-toric blowup of $\D$ with respect to $\Ha$.

\begin{thm}[type $\IV$ to type $\IV$]
\label{theorem: semi-toric typefour}
There is a natural finite extension of $\pi\colon\Gamma\bs \D\longrightarrow \widehat{\Gamma}\bs\widehat{\D}$ to semi-toric compactifications
\begin{equation*}
\pi\colon \overline{\Gamma\bs\D}^{\Sigma(\Ha)}\longrightarrow \overline{\widehat{\Gamma}\bs\widehat{\D}}^{\Sigma(\widehat{\Ha})}.
\end{equation*}
If $A$ satisfies condition \ref{condition: arithmetic quotient part}, the map is a normalization of its image.
\end{thm}
\begin{proof}
We first show the existence of $\pi\colon \overline{\Gamma\bs\D}^{\Sigma(\Ha)}\longrightarrow \overline{\widehat{\Gamma}\bs\widehat{\D}}^{\Sigma(\widehat{\Ha})}$ as a morphism between two projective varieties, then prove finiteness.

The subdomain is induced by $(W, \varphi)$. The isotropic lines and planes in $W$ are naturally viewed as boundary data of both $\D$ and $\widehat{\D}$. The conical locus of $\D$ is naturally embedded into that of $\widehat{\D}$.


Suppose $\sigma\in \Sigma(\Ha)$ does not correspond to a rational isotropic plane of $W$. Then we have a rational isotropic line $I$, such that $\sigma$ is contained in $C_{I,W,+}$ and intersects with $C_{I,W}$. For each $H$ containing $I$, the intersection $H\cap C_{I, W}$ being not empty is equivalent to $H\cap \D$ being not empty. Then there exists $\tau\in\Sigma(\widehat{\Ha})$ such that $\sigma=\tau\cap W$. We denote $\widehat{\sigma}$ to be the minimal element among all such $\tau$. Thus $\sigma=C_{I,W}\cap \widehat{\sigma}$, which implies $W_{\sigma}=V_{\widehat{\sigma}}\cap W$.

Let $\sigma\in \Sigma(\Ha)$ correspond to an isotropic plane $J$ contained in both $W$ and a hyperplane $H$. Suppose $v$ is a normal vector of $H$ and $v=w+w^\perp$ the decompostion in $V=W\oplus W^\perp$. We have $\varphi(v,v)<0$. The hyperplane $H$ intersects with $\D$ if and only if $\varphi(w,w)< 0$. Since the orthogonal complement of $w$ in $W_{\QQ}$ contains the isotropic plane $J$, we have either $\varphi(w,w)<0$ or $\varphi(w,w)=0$. Suppose the latter case happens, then $w\in J$ since otherwise $\langle J,w\rangle$ is an isotropic subspace of rank $3$ contained in $W_{\QQ}$, which is impossible. Thus in this case $H\supset J^{\perp}\cap W$. The above argument holds for any $H\in \Ha$ containing $\sigma$, hence $W_{\sigma}=V_{\sigma}\cap W$. In this case we also denote $\widehat{\sigma}=\sigma$.

For $\sigma=\{0\}\in \Sigma(\Ha)$, just take $\widehat{\sigma}=\{0\}\in \Sigma(\widehat{\Ha})$. Then for every $\sigma\in \Sigma(\Ha)$, we have a natural holomorphic map $\pi_{\sigma}\D\longrightarrow \pi_{\widehat{\sigma}}\widehat{\D}$ which is apparently injective. Taking union among $\sigma$, we have
\begin{equation*}
\coprod_{\sigma\in\Sigma(\Ha)}\pi_{\sigma}\D\longrightarrow\coprod_{\sigma\in\Sigma(\Ha)}\pi_{\widehat{\sigma}}\widehat{\D}
\hookrightarrow\coprod_{\tau\in\Sigma(\widehat{\Ha})}\pi_{\tau}\widehat{\D}
\end{equation*}
with the composition continuous. After taking quotients by the equivariant actions on both sides, we obtain a finite map between the boundary components. Actually, for those rational isotropic planes $J$, we obtain finite morphisms between Abelian torsors; for those rational isotropic lines $I$, we obtain finite morphisms between algebraic torus torsors. If $A$ satisfies condition \ref{condition: arithmetic quotient part}, then $\pi$ is generically injective by proposition \ref{proposition: normalization smooth case}, hence a normalization of its image.
\end{proof}

\begin{rmk}
The injectivity of $\coprod_{\sigma\in\Sigma(\Ha)}\pi_{\sigma}\D
\longrightarrow\coprod_{\tau\in\Sigma(\widehat{\Ha})}\pi_{\tau}\widehat{\D}$ is already known in \cite{looijenga2003typefour} (the paragraph after lemma 7.1).
\end{rmk}

For $\zeta\ne\overline{\zeta}$, we have ball $\B$ associated to $W=V_{\zeta}$. We next describe the semi-toric compactification of $\B$ with respect to $\Ha$. Here we identify elements in $\Ha$ with hypersurfaces in $W$. The cusp points correspond to isotropic lines $I$ in $W_K$. Let
\begin{equation*}
j(I)=(\underset{H\in \Ha, H\supset I}{\bigcap} H)\cap I_W^{\perp}
\end{equation*}
and $\pi_I\colon \PP(W)-\PP(j(I))\to \PP(W/j(I))$. Define
\begin{equation*}
\overline{\XX}^{j}=\Gamma\bs (\B\sqcup \coprod_I \pi_{j(I)}\B).
\end{equation*}
It naturally maps to the Baily-Borel compatification. The boundary component over each cusp point is an abelian torsor modeled over the vector space $I_W^\perp /j(I)$ quotient by a lattice.

\begin{thm}[ball to type $\IV$]
There is a natural finite extension of $\pi\colon\Gamma\bs \B\longrightarrow \widehat{\Gamma}\bs\widehat{\D}$ to semi-toric compactifications
\begin{equation*}
\pi\colon \overline{\Gamma\bs\B}^j\longrightarrow \overline{\widehat{\Gamma}\bs\widehat{\D}}^{\Sigma(\widehat{\Ha})}.
\end{equation*}
If $A$ satisfies condition \ref{condition: arithmetic quotient part}, the map is a normalization of its image.
\end{thm}
\begin{proof}
If $K$ is not a quadratic extension of $\QQ$, then $\XX$ is compact and the theorem holds. Now assume that $K$ is a quadratic extension of $\QQ$. Namely, $K=\QQ(\sqrt{-D})$ for certain positive integer $D$. Take any isotropic line $I$ in $W_K$. Suppose a nonzero generator of $I$ is $e+\sqrt{-D}f$, where $e,f\in V_{\QQ}$. Then $\varphi(e+\sqrt{-D}f, e-\sqrt{-D}f)=0$. From lemma \ref{lemma: W isotropic} we have $\varphi(e+\sqrt{-D}f,e+\sqrt{-D}f)=0$. This implies that $J=\langle e,f\rangle$ is an isotropic plane in $V_{\QQ}$.

We claim that $j(I)=W\cap V_J$. Take $H\in\widehat{\Ha}$ with orthogonal vector $v\in V_{\QQ}$. Under the orthogonal decomposition $V_{K}=W_K\oplus \overline{W_K}\oplus V^{\prime}$, we can decompose $v$ as $v=v_W+\overline{v_W}+v^{\prime}$. Then $\varphi(\mathrm{Re}(v_W),J)=0$. From lemma \ref{lemma: W isotropic} we have $\varphi(v_W, I)=0$. Therefore, $\varphi(\mathrm{Im}(v_W),I)=0$ and hence $\varphi(\mathrm{Im}(v_W),J)=0$.

Since $(V_{\QQ},\varphi)$ has signature $(2,N)$, the orthogonal complement of $J$ in $V_{\QQ}$ is negative semi-definite. Thus $\varphi(\mathrm{Re}(v_W),\mathrm{Re}(v_W))\le 0$ and $\varphi(\mathrm{Im}(v_W),\mathrm{Im}(v_W))\le 0$. We then have
\begin{equation*}
\varphi(v_W,\overline{v_W})=\varphi(\rRe(v_W),\rRe(v_W))+\varphi(\rIm(v_W),\rIm(v_W))\le 0.
\end{equation*}
Suppose $\varphi(v_W,\overline{v_W})<0$, then $H\cap \B\ne \varnothing$. Suppose $\varphi(v_W,\overline{v_W})=0$, then $v_W$ is an isotropic line in $W_K$. The vectors $\rRe(v_W)$ and $\rIm(v_W)$ in $V_{\QQ}$ are then isotropic. These two vectors are orthogonal to $J$, hence they belong to $J$. We deduce that $H\supset I_W^{\perp}$. By the definition of $j(I)$ and $V_J$, we conclude the claim.

We now have naturally an injective map $\pi_{j(I)}\B\longrightarrow \pi_J\widehat{\D}$. Taking the union among those isotropic lines $I$, we have an injective map
\begin{equation*}
\B\sqcup\coprod_I\pi_{j(I)}\B\hookrightarrow \coprod_{\sigma\in\Sigma(\widehat{\Ha})}\pi_{\sigma}\widehat{\D}.
\end{equation*}
After taking quotients by the equivariant actions on both sides, we obtain a morphism $\pi\colon \overline{\Gamma\bs\B}^j\longrightarrow\overline{\Gamma\bs\widehat{\D}}^{\Sigma(\widehat{\Ha})}$. Actually, the restriction of this $\pi$ to the boundary component corresponding to $I$ is a finite morphism between Abelian torsors. We conclude that there is natural extension $\pi\colon \overline{(\Gamma\bs\B)}^j\longrightarrow \overline{(\widehat{\Gamma}\bs\widehat{\D})}^{\Sigma(\widehat{\Ha})}$ which is a finite morphism between projective varieties. If $A$ satisfies condition \ref{condition: arithmetic quotient part}, this $\pi$ is generically injective, hence a normalization of its image.
\end{proof}

\begin{rmk}
By proposition \ref{proposition: ball or type four prime order}, a non-symplectic prime-order automorphism of a smooth cubic fourfold has order 2 or 3. This is an evidence for us to conjecture that for all balls arising from symmetric cubic fourfolds, the corresponding field $K$ is either $\QQ(\sqrt{-1})$ or $\QQ(\sqrt{-3})$.
\end{rmk}

\subsection{Main theorem}
\label{subsection: main theorem in appendix}
In this section, we first describe the construction of Looijenga compactification $\overline{\XX}^{\Ha}$ of $\XX^{\circ}\coloneqq \XX-\Gamma\bs\Ha$. We need to successively blow up non-empty intersections of components of $\Gamma\bs\Ha$, and then contract the strict transformations of $\Gamma\bs\Ha$ via a natural accociated semi-ample line bundle on the blowup. We then prove existence and finiteness of morphism between Looijenga compactifications on both sides of $\XX\longrightarrow\widehat{\XX}$.

The blow-up and blow-down constructions with respect to hyperplane arrangement in any normal analytic variety with a properly given line bundle are discussed in the first 3 sections in \cite{looijenga2003ball}. Looijenga applied this general theory to $(\overline{\XX}^{\Sigma(\Ha)},\Gamma\bs\Ha,\cL)$, where $\XX$ is either arithmetic quotient of type $\IV$ domain $\D$ or ball $\B$, and $\cL$ is the natural automorphic line bundle. See theorem 5.7 in \cite{looijenga2003ball} and theorem 7.4 in \cite{looijenga2003typefour}.

The blow-up and blow-down constructions before quotient by the arithmetic groups (and the Looijenga compactification can then be obtained by the last modified space quotient by the arithmetic group). We now describe this. Denote $\PO(\Ha)$ to be the set of nonempty intersections of elements in $\Ha$ as hyperplanes in $\D$ (or $\B$). Let $L\in \PO(\Ha)$ also denote its closure in $\D^{\Sigma}$ (or $\B^j$). Denote $c(L)\coloneqq \codim(L)-1$.

We first look at the semi-toric compactification $\D^{\Sigma}$ of $\D$. Denote $(\D^{\Sigma})^{\circ}$ to be the arrangement complement of $\Ha$ in $\D^{\Sigma}$. Choose $L\in \PO(\Ha)$ a minimal member. Blowing up along $L$ replaces $L$ by the projectivization of its normal bundle, which is isomorphic to $L\times\PP^{c(L)}$. The modified space, denoted by $\Bl_L(\D^{\Sigma})$, has natural topology, arrangement (the strict transform of the previous one) and automorphic line bundle. The strict transforms of those hypersurfaces passing through $L$ form a hyperplane arrangement in $\PP^{c(L)}$, and we denote the complement by $(\PP^{c(L)})^{\circ}$. The complement of the new arrangement in $\Bl_L(\D^{\Sigma})$ is the disjoint union $(\D^{\Sigma})^{\circ}\sqcup L\times (\PP^{c(L)})^{\circ}$. After blowing up successively until hypersurfaces disjoint, we obtain the final blowup $\widetilde{\D}$. This is a disjoint union of $(\D^{\Sigma})^{\circ}$ with $L\times (\PP^{c(L)})^{\circ}$ for all such minimal $L$ appearing in each step.

Now we can contract $L\times (\PP^{c(L)})^{\circ}$ along the direction of $L$ for all such $L$, and obtain $\D^*$ with natural quotient topology. Set-theoretically, $L\times (\PP^{c(L)})^{\circ}$ is contracted to $(\PP^{c(L)})^{\circ}$. We have the following discription (see \cite{looijenga2003typefour}):
\begin{equation}
\label{equation: components type four}
\D^*=\coprod_{L\in\PO(\Ha)}\pi_L\D^{\circ}\sqcup\coprod_{\sigma\in\Sigma(\Ha)}\pi_{\sigma}\D^{\circ}.
\end{equation}
Notice that for $\sigma$ being the vertex, $\pi_{\sigma}$ is identity and $\pi_{\sigma}\D^{\circ}=\D^{\circ}$.

The spaces $\D^{\Sigma}, \widetilde{\D}, \D^*$ constructed above all have natural ringed space structure. Namely, we have the structure sheaves consisting of continuous functions with analytic restriction to each stratum. The group $\Gamma$ naturally acts on those ringed spaces respecting the stratification. The topological quotient space $\overline{\XX}^{\Ha}\coloneqq\Gamma\bs\D^*$ has normal analytic structure respecting the stratification, see \cite{looijenga2003typefour} (theorem 7.4). According to the Riemann extension theorem, the quotient ringed space structure and the analytic structure on $\overline{\XX}^{\Ha}$ coincide.

For the case of ball, parallel argument gives $\widetilde{\B}$ and $\B^*$. We have:
\begin{equation*}
\B^*=\B^{\circ}\sqcup\coprod_{L\in\PO(\Ha)}\pi_L\B^{\circ}\sqcup\coprod_I \pi_{j(I)}\B^{\circ}.
\end{equation*}
This also has natural ringed structure, and $\overline{\XX}^{\Ha}\cong\Gamma\bs\B^*$ as analytic spaces.

\begin{thm}[Main Theorem]
\label{theorem: normalization looijenga compactification}
There is a natural finite extension of $\pi\colon\Gamma\bs (\D-\Ha) \longrightarrow \widehat{\Gamma}\bs(\widehat{\D}-\widehat{\Ha})$ to Looijenga compactifications
\begin{equation*}
\pi\colon \overline{\Gamma\bs\D}^{\Ha}\longrightarrow \overline{\widehat{\Gamma}\bs\widehat{\D}}^{\widehat{\Ha}}.
\end{equation*}
If $A$ satisfies condition \ref{condition: arithmetic quotient part}, the map is a normalization of its image. The same result holds for ball quotients.
\end{thm}
\begin{proof}
From theorem \ref{theorem: semi-toric typefour}, we have natural morphisms from $\D^{\Sigma}$ to $\widehat{\D}^{\Sigma}$. Near each boundary component, there is a contraction map from a neighborhood to the boundary itself. The arrangement in total space is the pullback of smooth arrangement on the boundary. According to the map defined near the boundary components, we know that any $H\in \widehat{\Ha}$ not intersecting with $\D$ is still away from $\D^{\Sigma}$ after taking its closure. From corollary 7.15 in chapter II in \cite{hartshorne1977algebraic}, we have injective map $\widetilde{\D}\longrightarrow\widetilde{\widehat{\D}}$ respecting the ringed space structures. Notice that the automorphic line bundle on $\D^{\Sigma}$ is the pull back of that on $\widehat{\D}^{\Sigma}$, hence we have an injective map on the strata $L\times (\PP^{c(L)})^{\circ}$ to $\widehat{L}\times (\PP^{c(\widehat{L})})^{\circ}$ which is linear on the second component. Here $\widehat{L}$ is a minimal member used in certain step of the successive blow-up construction of $\widehat{\D}$, and $L$ is the induced member on the smaller subspace by intersecting with $\widehat{L}$. After blowing down, we have a natural injective map $\D^*\longrightarrow \widehat{\D}^*$ respecting the ringed space structures.

This morphism descends to a morphism $\pi\colon \Gamma\bs\D^*\longrightarrow\widehat{\Gamma}\bs\widehat{\D}^*$, still in the category of ringed spaces. We then have an analytic morphism $\pi\colon\overline{\XX}^{\Ha}\longrightarrow\overline{\widehat{\XX}}^{\widehat{\Ha}}$. This analytic morphism extends $\pi\colon \XX^{\circ}\longrightarrow\widehat{\XX}^{\circ}$, and sends boundary strata into boundary strata. Combining with theorem \ref{theorem: semi-toric typefour}, the extended morphism $\pi$ here is finite. If $A$ satisfies condition \ref{condition: arithmetic quotient part}, it is generically injective and hence a normalization of its image.


The same argument also holds for ball.
\end{proof}





\end{appendix}

\noindent\textbf{Acknowledgement}: The first author is supported by Harvard University and University of Pennsylvania. He thanks his advisor Prof. S.-T. Yau for his constant support.

The second author is supported by Yau Mathematical Sciences Center at Tsinghua University. He is visiting Stony Brook University while this work is finished. He thanks Tsinghua University for funding his visit to Stony Brook, and thanks Stony Brook Mathematics Department for providing a great academic environment. He thanks his Ph.D. advisor Eduard Looijenga for his support and encouragement.

We thank Eduard Looijenga for stimulating discussion, especially on the appendix. We thank Radu Laza for helpful conversation. The second author thanks Gregory Pearlstein for helpful comments and for the invitation of a talk on this work. We also had useful communications with Ruijie Yang and Zheng Zhang about this work.
\bibliography{reference}
\bibliographystyle{alpha}

\Addresses
\end{document}